\documentclass{amsart}
\usepackage{amsmath, amsthm, amssymb, amscd, eucal, epsfig}
\usepackage{pinlabel}

\DeclareMathOperator{\PSL}{\mathrm{PSL}}

\input xy
\xyoption{all}

\usepackage{hyperref}

\begin{document}

\newtheorem{theorem}{Theorem}[subsection]
\newtheorem{lemma}[theorem]{Lemma}
\newtheorem{corollary}[theorem]{Corollary}
\newtheorem{conjecture}[theorem]{Conjecture}
\newtheorem{proposition}[theorem]{Proposition}
\newtheorem{question}[theorem]{Question}
\newtheorem{problem}[theorem]{Problem}
\newtheorem*{main_surface_theorem}{Surfaces in Random Groups~\ref{theorem:surface_random_group}}
\newtheorem*{thin_fatgraph_theorem}{Thin Fatgraph Theorem~\ref{theorem:thin_fatgraph}}
\newtheorem*{one_relator_theorem}{Random One Relator Theorem~\ref{theorem:surface_one_relator}}
\newtheorem*{surface_subgroup_question}{Surface Subgroup Question}
\newtheorem*{claim}{Claim}
\newtheorem*{criterion}{Criterion}
\theoremstyle{definition}
\newtheorem{definition}[theorem]{Definition}
\newtheorem{construction}[theorem]{Construction}
\newtheorem{notation}[theorem]{Notation}
\newtheorem{convention}[theorem]{Convention}
\newtheorem*{warning}{Warning}

\theoremstyle{remark}
\newtheorem{remark}[theorem]{Remark}
\newtheorem{example}[theorem]{Example}
\newtheorem*{case}{Case}

\def\id{\text{id}}
\def\Id{\text{Id}}
\def\1{{\bf{1}}}
\def\p{{\mathfrak{p}}}
\def\H{\mathbb H}
\def\Z{\mathbb Z}
\def\R{\mathbb R}
\def\N{\mathbb N}
\def\C{\mathbb C}
\def\F{\mathbb F}
\def\P{\mathbb P}
\def\Q{\mathbb Q}
\def\E{{\mathcal E}}
\def\D{{\mathcal D}}
\def\I{{\mathcal I}}

\def\CAT{\textnormal{CAT}}

\def\tra{\textnormal{tr}}
\def\length{\textnormal{length}}

\newcommand{\marginal}[1]{\marginpar{\tiny #1}}

\title{Random groups contain surface subgroups}
\author{Danny Calegari}
\address{Department of Mathematics \\ University of Chicago \\
Chicago, Illinois, 60637}
\email{dannyc@math.uchicago.edu}
\author{Alden Walker}
\address{Department of Mathematics \\ University of Chicago \\
Chicago, Illinois, 60637}
\email{akwalker@math.uchicago.edu}

\date{\today}

\begin{abstract}
A random group contains many quasiconvex surface subgroups. 
\end{abstract}

\maketitle

\section{Introduction}

Gromov famously asked the following:
\begin{surface_subgroup_question}
Let $G$ be a one-ended hyperbolic group. Does $G$ contain a subgroup isomorphic 
to the fundamental group of a closed surface with $\chi<0$?
\end{surface_subgroup_question}
Beyond its intrinsic appeal, and its obvious connections to the Virtual Haken Conjecture
in 3-manifold topology
(now a theorem of Agol \cite{Agol_VHC}), one reason 
Gromov was interested in this question was the hope that such surface subgroups could be used as 
essential structural components of hyperbolic groups \cite{Gromov_communication}. Our interest in
this question is stimulated by a belief that surface groups (not necessarily closed) 
can act as a sort of ``bridge'' between hyperbolic geometry and symplectic geometry
(through their connection to causal structures, quasimorphisms, stable commutator length, etc.).

Despite receiving considerable attention the Surface Subgroup Question is wide open in general, 
although in the specific case of hyperbolic 3-manifold groups it
was positively resolved by Kahn--Markovic \cite{Kahn_Markovic}. 
The main results of our paper may be summarized by saying that we show that Gromov's
question has a positive answer for {\em most} (hyperbolic) groups. In fact, the 
``executive summary'' says that
\begin{enumerate}
\item{{\em most groups} contain (many) surface subgroups;}
\item{these surface subgroups are {\em quasiconvex} --- i.e.\/ their intrinsic and
extrinsic geometry is uniformly comparable on large scales; and}
\item{these surface subgroups can be {\em constructed}, and their
properties {\em certified} quickly and easily.}
\end{enumerate}
Here ``most groups'' is a proxy for {\em random groups} in Gromov's few relators or
density models, to be defined presently.

\medskip

In \cite{Gromov_asymptotic}, \S~9 (also see \cite{Ollivier}), Gromov introduced the
notion of a {\em random group}. In fact, he introduced two such models: the {\em few relators}
model and the {\em density} model. In either model one first fixes a free group $F_k$
of rank $k\ge 2$ and a free generating set $x_1,\cdots,x_k$, and adds $\ell$
random relators of some fixed length $n$. In one model $\ell$ is a constant, independent of $n$.
In the other model $\ell = (2k-1)^{Dn}$ where now $D$ is constant, independent of $n$.
Explicitly:

\begin{definition}[Few relators model]
A random $k$-generator $\ell$-relator group at length $n$ is a group defined by a
presentation
$$G:=\langle x_1,\cdots, x_k \; | \; r_1,\cdots, r_\ell \rangle$$
where the $r_i$ are chosen randomly (with the uniform distribution) and independently from the
set of all cyclically reduced cyclic words of length $n$ in the $x_i^{\pm}$.
\end{definition}

\begin{definition}[Density model]
A random $k$-generator group at density $D$ (for some $0<D<1$) 
and at length $n$ is a group defined by a
presentation
$$G:=\langle x_1,\cdots, x_k \; | \; r_1,\cdots, r_\ell \rangle$$
where $\ell = (2k-1)^{Dn}$, and where the $r_i$ are chosen randomly (with the uniform distribution)
and independently from the set of all cyclically reduced cyclic words of length $n$ in the 
$x_i^{\pm}$.
\end{definition}

Thus properly speaking, either model defines a 
{\em probability distribution} on finitely presented groups (in fact, on finite
presentations) depending on constants $k,\ell,n$ in the few relators model, or on
$k,D,n$ in the density model.

If one is interested in a particular property of finitely presented groups, then one can
compute for each $n$ the probability that a random group as above has the desired property.
If this probability goes to $1$ as $n$ goes to infinity, then one says that a random $k$-generator group
(with $\ell$ relators; or at density $D$) has the given property {\em with overwhelming
probability}.
 
Gromov showed that at any fixed density $D>1/2$ random groups are trivial or isomorphic to $\Z/2\Z$, 
whereas at density
$D<1/2$ they are infinite, hyperbolic, and two-dimensional (with overwhelming
probability), and in fact the ``random presentation'' determined as above is aspherical. 
Later, Dahmani--Guirardel--Przytycki
\cite{Dahmani_Guirardel_Przytycki} showed that random groups at any density $D<1/2$ 
are one-ended and do not split, and therefore
(by the classification of boundaries of hyperbolic 2-dimensional groups),
have a Menger sponge as a boundary.

Random groups at density $D<1/6$ are known to be cubulated (i.e.\/ are equal to
the fundamental groups of nonpositively curved compact cube complexes), 
and at density $D<1/5$ to
act cocompactly (but not necessarily properly) on a $\CAT(0)$ cube complex, by Ollivier--Wise
\cite{Ollivier_Wise}. On the other hand, groups at density $1/3 < D < 1/2$ have property $(T)$,
by Zuk \cite{Zuk} (further clarified by Kotowski--Kotowski \cite{Kotowski_Kotowski}), 
and therefore cannot act on a $\CAT(0)$ cube complex without a global fixed
point. A one-ended hyperbolic cubulated group contains a one-ended graph of free groups
(see \cite{Calegari_Wilton}, Appendix~A; this depends on work of Agol \cite{Agol_VHC}), 
and Calegari--Wilton \cite{Calegari_Wilton} show
that a {\em random} graph of free groups (i.e.\/ a graph of free groups with random
homomorphisms from edge groups to vertex groups) contains a surface subgroup. Thus one might
hope that a random group at density $D<1/6$ should contain a graph of free groups that is
``random enough'' so that the main theorem of \cite{Calegari_Wilton} can be applied, 
and one can conclude that there is a surface subgroup.

Though suggestive, there does not appear to be an easy strategy to flesh out this idea. 
Nevertheless in this paper we are able to show directly that at {\em any} density $D<1/2$ a random group 
contains a surface subgroup (in fact, many surface subgroups). 

We give three proofs of this theorem, valid at different densities, with the final proof giving
any density $D<1/2$. Theorem~\ref{theorem:surface_one_relator} is valid for one-relator
groups (informally $D=0$), Lemma~\ref{lemma:density_convexity} gives $D<2/7$, while 
our main Theorem~\ref{theorem:surface_random_group} gives $D<1/2$. Explicitly, we show:

\begin{main_surface_theorem}
A random $k$-generator group at any density $D<1/2$ and length $n$
contains a surface subgroup with 
probability $1-O(e^{-n^c})$. In fact, it contains $O(e^{n^c})$ surfaces of genus $O(n)$.
Moreover, these surfaces are quasiconvex.
\end{main_surface_theorem}

This state of affairs is summarized in Figure~\ref{density_picture}. 
A modification of the construction (see Remark~\ref{remark:homologically_essential})
shows that the surface subgroups can be taken to be homologically essential.

\begin{figure}[htpb]
\labellist
\small\hair 2pt
\pinlabel $0$ at 40 100
%\pinlabel $\frac{1}{12}$ at 106.5 100
\pinlabel $\frac{1}{6}$ at 173 100
\pinlabel $\frac{1}{5}$ at 200 100
\pinlabel $\frac{2}{7}$ at 268 100
\pinlabel $\frac{1}{3}$ at 306 100
\pinlabel $\frac{1}{2}$ at 439.5 100
\pinlabel cubulated at 560 97.5
\pinlabel acts\;on\;$\CAT(0)$ at 579 72
\pinlabel property\;$(T)$ at 570 49
\pinlabel surface\;subgroup at 585 24
\pinlabel \ref{theorem:surface_one_relator} at 41 0
\pinlabel \ref{lemma:density_convexity} at 268 0
\pinlabel \ref{theorem:surface_random_group} at 439.5 0
\endlabellist
\centering
\includegraphics[scale=0.6]{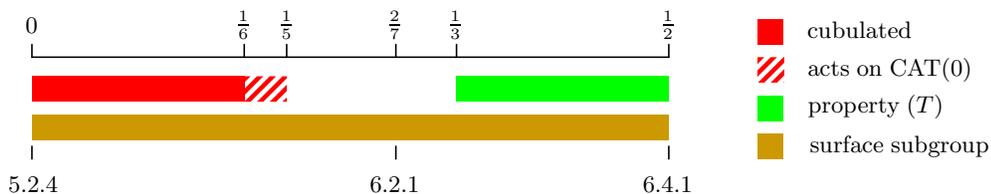}
\caption{Random groups at different densities}\label{density_picture}
\end{figure}

Along the way we prove some results of independent interest. The first of these (and the most
technically involved part of the paper) is the {\em Thin Fatgraph Theorem}, which says that
a ``sufficiently random'' homologically trivial collection of cyclic words $\Gamma$ in a free group satisfies
a strong combinatorial property: it can be realized as the oriented
boundary of a trivalent fatgraph in which
{\em every} edge is longer than some prescribed constant. This theorem is actually proved in
a {\em relative} version, where after having realized a collection of subwords $\Gamma' \subset \Gamma$
as the oriented boundary of a {\em partial trivalent fatgraph} (i.e.\/ a fatgraph with 3-valent
interior vertices and 1-valent ``boundary'' vertices), the remainder $\Gamma'':=\Gamma - \Gamma'$ can be
thought of as a collection of {\em tagged} cyclic words, where the tags indicate the boundary data
(i.e.\/ the way in which $\Gamma''$ lies inside $\Gamma$). Precise definitions of these terms are
given in \S~\ref{subsection:tags}.

\begin{thin_fatgraph_theorem}
For all $L>0$, for any $T\gg L$ and any $0<\epsilon \ll 1/T$, there is an $N$ depending only on $L$
so that if $\Gamma$ is a homologically trivial collection of tagged loops such that for each loop
$\gamma$ in $\Gamma$:
\begin{enumerate}
\item{no two tags in $\gamma$ are closer than $4L$;}
\item{the density of the tags in $\gamma$ is of order $o(\epsilon)$;}
\item{$\gamma$ is $(T,\epsilon)$-pseudorandom;}
\end{enumerate}
then there exists a trivalent fatgraph $Y$ with every edge of length at least $L$ so that
$\partial S(Y)$ is equal to $N$ disjoint copies of $\Gamma$.
\end{thin_fatgraph_theorem}
If the rank of the group is $2$, we can take $N=1$ above; otherwise we can take $N=20L$.

The Thin Fatgraph Theorem strengthens one of the main technical theorems underpinning 
\cite{Calegari_Walker_LP} and \cite{Calegari_Wilton}, and can be thought of as a kind of 
$L^\infty$ theorem whose $L^1$ version (with optimal constants) is the main theorem of 
\cite{Calegari_Walker_RR}. If $r$ is a long random relator, the Thin Fatgraph Theorem lets
us build a surface whose boundary consists of a small number
of copies of $r$ and $r^{-1}$. By plugging in a disk along each boundary component, we obtain
a closed surface in the one-relator group $\langle F_k\; | \; r \rangle$. If the surface
is built correctly, it can be shown to be $\pi_1$-injective, with high probability.
This is one of the most subtle parts of the construction, and ensuring that the surfaces
we build are $\pi_1$-injective at this step depends on the existence of a so-called 
{\em Bead Decomposition} for $r$; see Lemma~\ref{lemma:bead_decomposition}. Thus
we obtain the Random One Relator Theorem, whose statement is as follows:
\begin{one_relator_theorem}
Fix a free group $F_k$ and let $r$ be a random cyclically reduced word of length $n$. Then
$G:=\langle F_k\; | \; r \rangle$ contains a 
surface subgroup with probability $1-O(e^{-n^c})$.
\end{one_relator_theorem}

The surfaces stay injective as more and more relators are added (in fact, these are the surfaces
referred to in the main theorem) so this shows that random groups in the few relators model
also contain surface subgroups for any fixed $\ell>0$, with high probability.

There is an interesting tension here: the fewer relators, the harder it is to build a surface group,
but the easier it is to show that it is injective. This suggests looking for surface subgroups in
an arbitrary one-ended hyperbolic group at a very specific ``intermediate'' scale, perhaps
at the scale $O(\delta)$ where $\delta$ is the constant of hyperbolicity with respect to an 
``efficient'' (e.g. Dehn) presentation.

\medskip

We conclude this introduction with three remarks.

First: it is worth spelling out some similarities and differences between our work 
and the breakthrough work of Kahn--Markovic \cite{Kahn_Markovic}. The Kahn--Markovic argument 
depends crucially on the structure of hyperbolic 3-manifold 
groups as lattices in the semisimple Lie group
$\PSL(2,\C)$. By contrast, in this paper we are concerned with much more combinatorial classes of hyperbolic
groups. Nevertheless, one common point of contact
is the use of probability theory to {\em construct}
surfaces, and the use of (hyperbolic) geometry to {\em certify} them as injective. In particular, because
our surfaces are certified as injective by local methods, they end up being quasiconvex. It is an
interesting question to identify the class of hyperbolic groups which contain
non-quasiconvex (yet injective) surface subgroups (hyperbolic 3-manifold groups are now known
to contain such groups since they are virtually fibered, again by Agol \cite{Agol_VHC}).

Second: a large part of the difficulty in the proof of the Thin Fatgraph Theorem
arises because we insist on building {\em oriented} surfaces. The advantage of this is that
when our random groups $G$ have nontrivial $H_2$ (which happens whenever $D>0$ in the density model) the
injective surfaces we construct can be chosen to be {\em homologically essential} in $G$.
On the other hand, for the reader who is interested only in the existence of closed
surface subgroups in $G$, the proof of the Thin Fatgraph Theorem can be considerably simplified. We
explain this at the end of \S~\ref{section:good_pants}.

Third: the reader who is not already invested in the theory of random groups might complain
that the few relators and density models seem rather special, insofar as the random relators are sampled
from an especially simple probability distribution (i.e.\/ the uniform distribution). One may consider
a variation on the construction of a random group by fixing $F_k$ and a stationary Markov process
of entropy $\log(\lambda)>0$ which successively generates the letters of reduced words in $F_k$, 
and define a random group
at density $D$ and length $n$ to be obtained by adding $\lambda^{nD}$ words of length $n$ as
relators, each generated independently by the Markov process. Providing the Markov process is
ergodic and has {\em full support} --- i.e.\/ providing that every finite reduced word $\sigma$
has a positive probability of being generated --- a random group in this 
model will contain surface groups with overwhelming
probability for any $D<1/2$. If we further assume that for a long random string generated by the Markov
process and for any $\sigma$ as above the expected number of copies of $\sigma$ and of 
$\sigma^{-1}$ are equal, then the surface subgroups can be chosen to be homologically essential. 

\subsection{Acknowledgments}

We would like to thank Misha Gromov, John Mackay, Yann Ollivier, Piotr Przytycki, Henry Wilton
and the anonymous referee. 
We also would like to acknowledge the use of Colin Rourke's {\tt pinlabel} program, and
Nathan Dunfield's {\tt labelpin} program to help add the (numerous!) labels to the figures.
Danny Calegari was supported by NSF grant DMS 1005246, 
and Alden Walker was supported by NSF grant DMS 1203888.

\section{Background}

In this section we describe some of the standard combinatorial language that we use in
the remainder of the paper. Most important is the notion of {\em foldedness} for a map
between graphs, as developed by Stallings \cite{Stallings}. We also recall some standard
elements of the theory of small cancellation, which it is convenient to cite at certain
points in our argument, though ultimately we depend on a more flexible version of small
cancellation theory developed by Ollivier \cite{Ollivier_small} specifically for application to
random groups (his results are summarized in \S~\ref{subsection:Ollivier_results}).

\subsection{Fatgraphs and foldedness}

\begin{definition}
Let $X$ and $Y$ be graphs. A map $f:Y \to X$ is {\em simplicial} if it
takes edges (linearly) to edges. It is {\em folded} if it is locally
injective.
\end{definition}

A folded map between graphs is injective on $\pi_1$. The terminology
of foldedness, and its first effective use as a tool in group theory, is
due to Stallings \cite{Stallings}.

\begin{definition}
A {\em fatgraph} is a graph $Y$ together with a choice of cyclic order
on the edges incident to each vertex. A fatgraph admits a canonical
{\em fattening} to a surface $S(Y)$ in which it sits as a spine (so that $S(Y)$
deformation retracts to $Y$) in such a way that the cyclic order of edges
coming from the fatgraph structure agrees with the cyclic order in which
the edges appear in $S(Y)$. A {\em folded fatgraph over $X$} is a
fatgraph $Y$ together with a folded map $f:Y \to X$.
\end{definition}

The case of most interest to us will be that $X$ is a rose associated to
a free generating set for a (finitely generated) free group $F$.

A folded fatgraph $f:Y \to X$ induces a $\pi_1$ injective map $S(Y) \to X$.
The deformation retraction $S(Y) \to Y$ induces an immersion $\partial S(Y) \to Y$,
and we may therefore think of $\partial S(Y)$ as a union of simplicial
loops. Under $f$ these loops map to immersed loops in $X$, corresponding
to conjugacy classes in $\pi_1(X)$.

Conversely, given a homologically trivial
collection of conjugacy classes $\Gamma$ in $\pi_1(X)$ represented
(uniquely) by immersed oriented loops in $X$, we may ask whether there is
a folded fatgraph $Y$ over $X$ so that $\partial S(Y)$ represents $\Gamma$
(by abuse of notation, we write $\partial S(Y) = \Gamma$). Informally,
we say that such a $\Gamma$ {\em bounds a folded fatgraph}.

\subsection{Small cancellation}

\begin{definition}
Let $G$ have a presentation 
$$G:=\langle x_1,\cdots, x_n\; | \; r_1,\cdots, r_s\rangle$$
where the $r_j$ are cyclically reduced words in the generators $x_i^{\pm}$.
A {\em piece} is a subword that appears in two different ways in the
relations or their inverses. A presentation satisfies the condition
$C'(\lambda)$ for some $\lambda$ if every piece $\sigma$ in some
$r_i$ satisfies $|\sigma|/|r_i| < \lambda$.
\end{definition}

\begin{remark}
Some authors use the notation $C'(\lambda)$ to indicate the weaker
inequality $|\sigma|/|r_i| \le \lambda$. This distinction will be irrelevant for us.
\end{remark}

Associated to a presentation there is a connected
2-complex $K$ with one vertex, one edge for each generator, and one
disk for each relation. The 1-skeleton $X$ for $K$ is a rose for the free
group on the generators. As is well-known, a group satisfying 
$C'(1/6)$ is hyperbolic, and (if no relator is a proper power)
the 2-complex $K$ is aspherical (so that the
group is of cohomological dimension at most 2).

\begin{definition}
Fix a group $G$ with a presentation complex $K$ and 1-skeleton $X$ as
above. A {\em surface over the presentation} is an oriented surface $S$ with the
structure of a cell complex together with a cellular map to $K$ which
is an isomorphism on each cell. The 1-skeleton $Y$ of the CW complex structure
on $S$ inherits the structure of a fatgraph from $S$ and its orientation,
and this fatgraph comes together with a map to $X$.
We say $S$ has a {\em folded spine} if $Y \to X$ is a folded fatgraph.
\end{definition}

If $G$ is a small cancellation group, a surface with a folded spine can
be certified as $\pi_1$-injective by the following combinatorial
condition.

\begin{definition}\label{definition:alpha_convex}
Let $G$ be a group with a fixed presentation, and let
$S$ be an oriented surface over the presentation with a folded spine $Y$.
We say $S$ is {\em $\alpha$-convex} (for some $\alpha>0$) with respect to the presentation
if for every immersed path $\gamma$ in $Y$ which is a subword in some
relation $r_i^{\pm}$ with $|\gamma|/|r_i|\ge \alpha$, we actually
have that $\gamma$ is contained in $\partial S(Y)$ (i.e.\/ it is in the boundary of a disk of
$S$).
\end{definition}

\begin{lemma}[Injective surface]\label{lemma:injective_surface}
Let $G$ be a group with a presentation satisfying $C'(1/6)$ and such that no
relator is a proper power; and let $S$
be an oriented surface over the presentation with a folded spine $Y$. 
If $S$ is $1/2$-convex then it is $\pi_1$-injective.
\end{lemma}
\begin{proof}
First we prove injectivity under the assumption that $S$ is $1/2$-convex. 
Suppose not, so that there is some essential loop in $\pi_1(S)$ which is
trivial in $G$. After a homotopy, we can assume this loop $\gamma$ is
{\em immersed} in $Y$. Since $Y$ is folded, the image of $\gamma$ in $X$
is also immersed; i.e.\/ it is represented by a cyclically reduced word
in the generators. Since by hypothesis $\gamma$ is trivial in $G$, there is
a van Kampen diagram with $\gamma$ as boundary. We may choose $\gamma$
and a diagram for which the number of faces is minimal.

The $C'(1/6)$ condition implies that there is a face $D$ in the diagram which
has at least $1/2$ of its boundary as a connected segment on $\gamma$; this is
sometimes called {\em Greedlinger's Lemma}.
Then the hypothesis implies that this segment is actually contained in the
boundary of a disk $D'$ of $S$. Since $G$ is $C'(1/6)$ it follows that 
$D'$ and $D$ bound the same relator in the same way, and we can therefore
push $\gamma$ across $D'$ to obtain a van Kampen diagram with fewer
faces and with boundary an essential loop in $S$ (homotopic to $\gamma$). 
But this contradicts the choice of van Kampen diagram, and this
contradiction shows that no such essential loop exists; i.e.\/ that
$\pi_1(S) \to G$ is injective.
\end{proof}

\begin{remark}
If $S$ is $\alpha$-convex for any fixed $\alpha < 1/2$, a similar argument shows
that $S$ is quasiconvex; since we shall prove quasiconvexity under more general
geometric hypotheses in Theorem~\ref{theorem:surface_random_group}, and since this fact is
not actually used in the paper, we do not justify this remark here.
\end{remark}

In the sequel we usually say that a map from $S$ to $K$ is {\em injective} to mean
that it is $\pi_1$-injective.

\section{Trivalent fatgraphs}

The purpose of this section is to prove the Thin Fatgraph Theorem~\ref{theorem:thin_fatgraph},
which implies that a (homologically trivial) collection of random cyclically reduced
words bounds a trivalent fatgraph with long edges (i.e.\/ in which {\em every} edge is
as long as desired).

For concreteness the theorem is stated not for random words but for (sufficiently) pseudorandom
words, and does not therefore really involve any probability theory. However the (obvious)
application in this paper is to random words, and words obtained from them by simple operations.

\subsection{Partial fatgraphs and tags}\label{subsection:tags}

We are going to build folded fatgraphs with prescribed boundary (i.e.\/
given $\Gamma$ we will build $Y$ with $\Gamma = \partial S(Y)$). In the
process of building these fatgraphs we deal with intermediate objects that we
call {\em partial fatgraphs} bounding part of $\Gamma$, and the part of $\Gamma$
that is not yet bounded by a partial fatgraph is a collection of {\em cyclic words
with tags}. This language is introduced in \cite{Calegari_Walker_LP}.

\begin{figure}[htpd]
\begin{center}
\labellist
\small
\pinlabel {$x$} at 173 105
\pinlabel {$Y$} at 128 108
\pinlabel {$u$} at 149 157
\pinlabel {$v$} at 146 16
\pinlabel {$z$} at 241 65
\pinlabel {$X$} at 247 185
\pinlabel {$w$} at 149 235
\pinlabel {$y$} at 53 183
\pinlabel {$Z$} at 54 62
\pinlabel {$x$} at 427 155
\pinlabel {$u$} at 410 172
\pinlabel {$Y$} at 391 159
\pinlabel {$v$} at 409 6
\pinlabel {$z$} at 423 102
\pinlabel {$X$} at 436 128
\pinlabel {$w$} at 408 238
\pinlabel {$y$} at 380 131
\pinlabel {$Z$} at 398 103
\pinlabel {$u$} at 668 121
\pinlabel {$v$} at 670 8
\pinlabel {$w$} at 667 240 
\endlabellist
\includegraphics[scale=0.47]{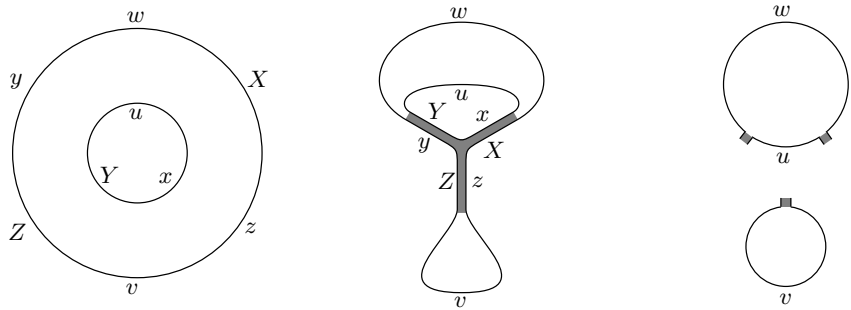}
\caption{Two cyclic words are partially paired along a partial fatgraph (the grey tripod); 
what is left is two cyclic words with three tags.}
\label{new_fat_with_tags}
\end{center}
\end{figure}

The (partial) fatgraphs will be built by taking disjoint pairs of segments in
$\Gamma$ with inverse labels (in $X$) and {\em pairing them} --- i.e.\/
associating them to opposite sides of an edge of the fatgraph. Once
all of $\Gamma$ is decomposed into such paired segments the fatgraph $Y$
will be implicitly defined.

A {\em partial fatgraph} is, abstractly, the data of a pairing of
some collection of disjoint pairs of segments in $\Gamma$. We imagine that
this partial fatgraph $Z$ has boundary $\partial S(Z)=:\Gamma'$ which is
a subset of $\Gamma$. The difference $\Gamma - \Gamma'$ is a collection
of paths whose endpoints are paired according to how they are paired in $\Gamma'$.
The result is therefore a collection of cyclic words $\Gamma''$, together with the data
of the ``germ'' of the partial fatgraph $Z$ at finitely many points. This extra
data we refer to as {\em tags}, and we call this collection $\Gamma''$ a 
collection of {\em cyclic words with tags}. 

\begin{example}
An example is illustrated in Figure~\ref{new_fat_with_tags}. Starting with two reduced
cyclic words $vzXwyZ$ and $uxY$ we pair the subwords $zX$, $xY$ and $yZ$ along the edges of
a tripod (as indicated in the figure) leaving ``tagged'' cyclic words $u\cdot w\, \cdot$ and $v\,\cdot$
as a remainder (in formulas the tags can be indicated by the punctuation character $\cdot$).
\end{example}

\subsection{Pseudorandomness}

Random (cyclically reduced) words enjoy many strong equidistribution properties, at
a large range of scales. For our purposes it is sufficient to have ``enough'' equidistribution
at a sufficiently large {\em fixed} scale. To quantify this we describe the condition of
{\em pseudorandomness}, and observe that random words are pseudorandom with high probability.

\begin{definition}\label{definition:pseudorandom}
Let $\Gamma$ be a cyclically reduced cyclic word in a free group $F_k$ with $k\ge 2$ generators.
We say $\Gamma$ is {\em $(T,\epsilon)$-pseudorandom} if the following is true: if we pick any
cyclic conjugate of $\Gamma$, and write it as a product of reduced words $w_i$ of length $T$ (and at
most one word $v$ of length $<T$) 
$$\Gamma: = w_1w_2w_3\cdots w_Nv$$
then for every reduced word $\sigma$ of length $T$ in $F_k$, there is an estimate
$$1-\epsilon \le \frac {\# \lbrace i \text{ such that } w_i = \sigma \rbrace} 
{N} \cdot (2k)(2k-1)^{T-1}\le 1+\epsilon$$
Similarly, we say that a collection of reduced words $w_i$ of length $T$ is 
{\em $\epsilon$-pseudorandom} if for every reduced word $\sigma$ of length $T$ in $F_k$ the
estimate above holds.
\end{definition}

\begin{lemma}[Random is pseudorandom]\label{lemma:pseudorandom}
Fix $T,\epsilon > 0$. Let $\Gamma$ be a random cyclically reduced word of length $n$. Then
$\Gamma$ is $(T,\epsilon)$-pseudorandom with probability $1-O(e^{-Cn})$.
\end{lemma}
\begin{proof}
This is immediate from the Chernoff inequality for finite Markov chains 
(see \S~\ref{subsection:independence} for a precise statement of the form the of 
Chernoff inequality we use, and for references).
\end{proof}

\subsection{Thin Fatgraph Theorem}

We now come to the main result in this section, the Thin Fatgraph Theorem. This says
that any (sufficiently) pseudorandom homologically trivial 
collection of tagged loops, with sufficiently few and
well-spaced tags, bounds a trivalent fatgraph with every edge as long as desired. Note that
every trivalent graph (with reduced boundary) is automatically folded.

This theorem can be compared with \cite{Calegari_Walker_LP}, Thm.~8.9 which says that 
random homologically trivial words bound 4-valent folded fatgraphs, with high probability; and \cite{Calegari_Walker_RR},
Thm.~4.1 which says that random homologically trivial words of length $n$ bound 
(not necessarily folded) fatgraphs whose {\em average} valence is arbitrarily close to 3, and whose
{\em average} edge length is as close to $\log(n)/2\log(2k-1)$ as desired (and moreover
this quantity is sharp). It would be very interesting
to prove (or disprove) that random homologically trivial words bound (with high probability)
trivalent fatgraphs in which {\em every} edge has length $O(\log(n))$, but this seems to require
new ideas.

\begin{theorem}[Thin Fatgraph]\label{theorem:thin_fatgraph}
For all $L>0$, for any $T\gg L$ and any $0<\epsilon \ll 1/T$, there is an $N$ depending only on $L$
so that if $\Gamma$ is a homologically trivial collection of tagged loops such that for each loop
$\gamma$ in $\Gamma$:
\begin{enumerate}
\item{no two tags in $\gamma$ are closer than $4L$;}
\item{the density of the tags in $\gamma$ is of order $o(\epsilon)$;}
\item{$\gamma$ is $(T,\epsilon)$-pseudorandom;}
\end{enumerate}
then there exists a trivalent fatgraph $Y$ with every edge of length at least $L$ so that
$\partial S(Y)$ is equal to $N$ disjoint copies of $\Gamma$.
\end{theorem}

The notation $T\gg L$ means ``for all $T$ sufficiently large depending on $L$'', and similarly
$0 < \epsilon \ll 1/T$ means ``for all $\epsilon$ sufficiently small depending on $T$''.
The {\em density} of tags is just the number of tags divided by the length of $\gamma$, and
the notation $o(\epsilon)$ just means something of negligible size compared to $\epsilon$.
The role of $N$ will become apparent at the last step, where some combinatorial condition can be
solved more easily over the rationals than over the integers (so that one needs to take a multiple
of the original chain in order to clear denominators). In fact, in rank 2 we can actually take $N=1$, 
and in higher rank we can take $N=20L$ (it is probably true that one can take $N=1$ always, but
this is superfluous for our purposes).

Except for the last step (which it must be admitted is quite substantial and takes up
almost half the paper), the argument is very close to that in \cite{Calegari_Walker_LP}.
For the sake of completeness we reproduce that argument here, explaining how to modify it
to control the edge lengths and valence of the fatgraph.

\subsection{Experimental results}\label{subsection:experimental_results}

Theorem~\ref{theorem:thin_fatgraph} asserts that long random words bound 
trivalent fatgraphs (up to taking sufficiently many disjoint copies).  However,
in order for the pseudorandomness to hold at scales required by the argument,
it is necessary to consider random words of enormous length; i.e.\/ on the order of a googol or
more. On the other hand, experiments show that even words of 
modest length bound trivalent fatgraphs with high probability.  
To keep our experiment simple, we considered only the condition of bounding a trivalent
graph, ignoring the question of whether the edges can all be chosen to be long.  

\begin{figure}[ht]
\begin{center}
\labellist
\small\hair 2pt
 \pinlabel {$1$} at 0 323
 \pinlabel {$0.1$} at -8 246
 \pinlabel {$0.01$} at -12 171
 \pinlabel {$0.001$} at -16 95
 \pinlabel {$0.0001$} at -23 18
 \pinlabel {$20$} at 95 -8
 \pinlabel {$40$} at 181 -8
 \pinlabel {$60$} at 265 -8
 \pinlabel {$80$} at 350 -8
 \pinlabel {$100$} at 435 -8
 \pinlabel {$120$} at 520 -8
\endlabellist
\includegraphics[scale=0.5]{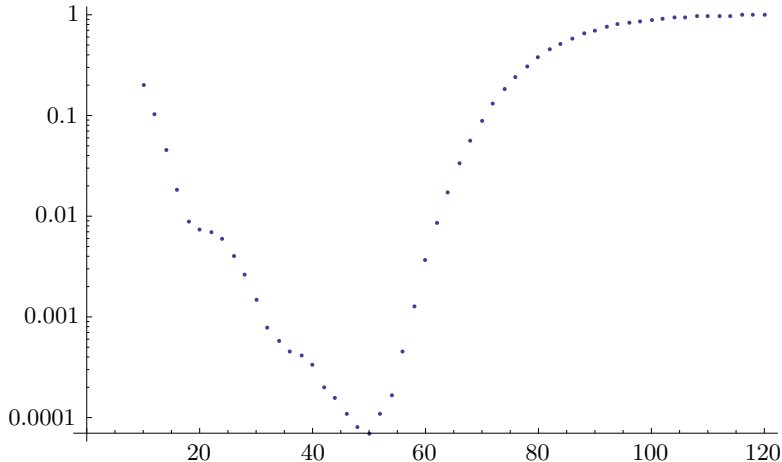}
\caption{The experimental log-fraction of random words of each length 
which bound a trivalent fatgraph.}
\label{figure:trivalent_experiment}
\end{center}
\end{figure}

In a free group of rank 3, we looked at between 100000 and 400000 cyclically reduced
homologically trivial words of each even length from 10 to 120. The proportion of such words 
that bound trivalent fatgraphs is plotted in
Figure~\ref{figure:trivalent_experiment}.  The vertical axis has a log-scale to show some interesting 
features of the data.  As one can see, bounding a trivalent fatgraph 
happens in practice for $n$ far below the purview of Theorem~\ref{theorem:thin_fatgraph}.
The curious local minimum at length $\sim 50$ is presumably a combinatorial artifact.

\subsection{Proof of the Thin Fatgraph Theorem}\label{subsection:thin_fatgraph}

We now give the proof of Theorem~\ref{theorem:thin_fatgraph}. The proof proceeds in several
steps. The first few steps are more probabilistic in nature. The last step is more
combinatorial and quite intricate, and is deferred to \S~\ref{section:good_pants}.

\medskip

Pick a $\gamma$ in $\Gamma$. Now, $\gamma$ is a cyclic word; starting at any letter we can
express it in the form 
$$\gamma= w_1w_2\cdots w_N v$$
where each $|w_i|=T$ and $N=\lfloor |\gamma|/T \rfloor$.
Since $\gamma$ is $(T,\epsilon)$-pseudorandom, the $w_i$ are very well equidistributed among
the reduced words of length $T$. Moreover, since by hypothesis (2)
the density of tags is of order $o(\epsilon)$, the proportion of the $w_i$ that contain a
tag is also of order $o(\epsilon)$. In the next step we restrict attention to the $w_i$ 
that do not contain a tag.

\subsubsection{Tall poppies}

Throughout the remainder of the proof we fix some $T'$ which is an odd multiple of $10L$
with $1000L < T-T' \le 2000L$ (in fact, something like $T-T'>4L$ is sufficient, but there is
no point in trying to optimize constants here). 
Note that we still have $T'\gg L$. For each $w_i$ we let $v_i$
be the initial subword of length $T'$. Note that the map which takes a reduced word
of length $T$ to its prefix of length $T'$ takes the uniform measure to a multiple of
the uniform measure, and therefore the $v_i$ are also $\epsilon$-pseudorandom.

The first step is to create a collection of {\em tall poppies}. We fix some $v:=v_i$ and
read the letters one by one. As we read along, we look for a pair of inverse subwords
$x,X$ each of length $10L$ and separated by a subsegment $y$ of length $40L$. Further we require
that the copy of $xyX$ should have the property that the $x$ and $X$ are maximal inverse
subwords at their given locations, so that the result of pairing creates reduced tagged cyclic
words. If the copy of $xyX$ is not too close to a tag of $\gamma$ (say, there is no tag within
a $10L$ neighborhood), we create some partial fatgraph by identifying $x$ to $X$;
this creates a {\em tall poppy} whose {\em stem} is $x$, and whose {\em flower} is $y$.
Once we find and create a tall poppy, we look for each subsequent tall poppy at successive
locations along $v$ subject to the constraint that adjacent tall poppies are separated by subwords
whose length is an even multiple of $10L$. Furthermore, we insist that the first  
tall poppy occurs at distance an even multiple of $10L$ from the start of $v$. See
Figure~\ref{tall_poppies} for an example; the ``dots'' in the figure indicate units of $10L$.

\begin{figure}[htpd]
\begin{center}
\includegraphics[scale=0.8]{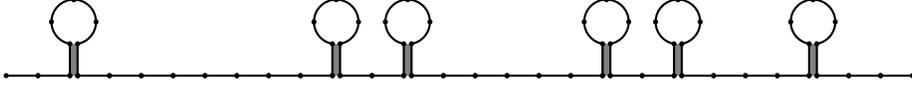}
\caption{A word $v$ of length $630L$ with $6$ tall poppies folded off}
\label{tall_poppies}
\end{center}
\end{figure}

For each $v_i$ we fold off tall poppies as above. The result of this
step is to create a partial fatgraph for each $\gamma$ consisting of some tagged loop
$\gamma'$ (which is obtained from $\gamma$ by cutting out all the $xyX$ subwords and identifying
endpoints) and a {\em reservoir} of flowers. Observe that every tagged cyclic
word of length $40L$ occurs as a flower, and the set of tagged flowers is
$\epsilon$-pseudorandom (conditioned on any compatible label on the tag).
Note that as remarked above, we are only restricting attention to $w_i$ that do not contain
a tag of $\gamma$, so the operation of creating a tall poppy will never produce two tags
that are too close together. 

Informally, we say that the reservoir contains an {\em almost equidistributed collection}
of tagged cyclic words of length $40L$. We can estimate the total number of flowers of each kind:
at each location that a flower might occur, we require two subwords of length $10L$ to be
inverse, which will happen with probability $(2k-1)^{-10L}$. The number of locations is roughly of
size $O(|\gamma|/10L)$. So the number of copies of each tagged loop in the reservoir is of size
$\delta\cdot|\gamma|$ (up to multiplicative error $1\pm \epsilon$) for some specific positive
$\delta>0$ depending only on $L$.

\subsubsection{Random cancellation}

After cutting off tall poppies, the $v_i$ become tagged words $v_i'$. Observe that the $v_i'$
have variable lengths (differing from $T'$ by an even multiple of $10L$) and have tags occurring
at some subset of the points an even multiple of $10L$ from the start. The main observation
to make is that the $\epsilon$-pseudorandomness of the $v_i$ 
propagates to $\epsilon$-pseudorandomness of the $v_i'$. That is, if $\sigma$ is a reduced word of
length $T'-m10L>0$ for some even $m$, then among the $v_i'$ of length $T'-m10L$, the proportion that
are equal to $\sigma$ is equal to $1/(2k)(2k-1)^{T'-m10L-1}$ up to a multiplicative error of size
$1\pm \epsilon$. This is immediate from the construction.

Recall that we chose $T'$ to be an {\em odd} multiple of $10L$. This means that when
we pair a segment $v_i'$ labeled $\sigma$ with a $v_j'$ labeled $\sigma^{-1}$ the tags of
$v_i'$ and $v_j'$ do not match up, and in fact any two tags are no closer than distance $10L$.
In fact, it is important that after pairing up inverse segments, the tagged loops that remain are
reduced, so we write each $v_i'$ in the form $l_i v_i'' r_i$ where each of $l_i,r_i$ has length
$5L$, and pair $v_i''$ with $v_j''$ for some $v_j'$ of the form $l_j v_j'' r_j$ 
where $v_j'' = (v_i'')^{-1}$, and the words
$l_ir_j$ and $r_il_j$ are reduced. By $\epsilon$-pseudorandomness, we can find such pairings of
all but $O(\epsilon)$ of the $v_i'$ in this way. Here, as in the previous subsection, we
do {\em not} pair $v_i'$ that contain one of the original tags of $\gamma$; since the
fraction of such $v_i'$ is $o(\epsilon)$ (again by hypothesis (2)), 
the error term can be absorbed into the $O(\epsilon)$ term.

Thus the result of this pairing is to produce a trivalent partial fatgraph with all edges of
length at least $5L$. Removing this from the $v_i'$ produces a collection of tagged loops $\gamma''$ with
$|\gamma''| = O(\epsilon\cdot |\gamma|)$.

\subsubsection{Cancelling $\Gamma''$ from the reservoir}

Let $\Gamma''$ be the union of all the $\gamma''$, and pool the reservoirs from
each $\gamma$ into a single reservoir.

Notice that by construction, and by hypothesis (1) of the theorem, 
no tagged loop in $\Gamma''$ has two tags closer than distance $4L$ (in fact, it is only the
original tags of $\gamma$ which might be as close to each other as $4L$; the tags arising
from tall poppies or by identifying the various $v_i'$ in pairs will all be distance at least $20L$
apart).

For each tagged loop $\nu$ in $\Gamma''$ we can build a copy of $\nu^{-1}$
out of finitely many flowers in the reservoir in such a way that the result of pairing
$\nu$ to this $\nu^{-1}$ is a trivalent partial fatgraph with all edges of length at least $L$, 
and the number of flowers that we need is proportional to $|\nu|/40L$. There is a
slight subtlety here, in that the length of each flower is $40L$, and the result
of partially gluing up a collection of cyclic words of even length always leaves an
even number of letters unglued. Fortunately, the assumption that $\Gamma$ is homologically
trivial implies that $|\Gamma|$ itself is even, and since each
flower also has an even number of letters, it follows that $|\Gamma''|$ is even.
A flower with the cyclic word $xyzY$ can be partially glued to produce two tagged
loops $x$ and $z$, and if $x$ and $z$ are odd, each can be used to contribute to
a copy of some $\nu^{-1}$ of odd total length. Since the number of odd $|\nu|$ is
even, all of $\Gamma''$ can be cancelled in this way.

\begin{figure}[htpd]
\begin{center}
\includegraphics[scale=0.4]{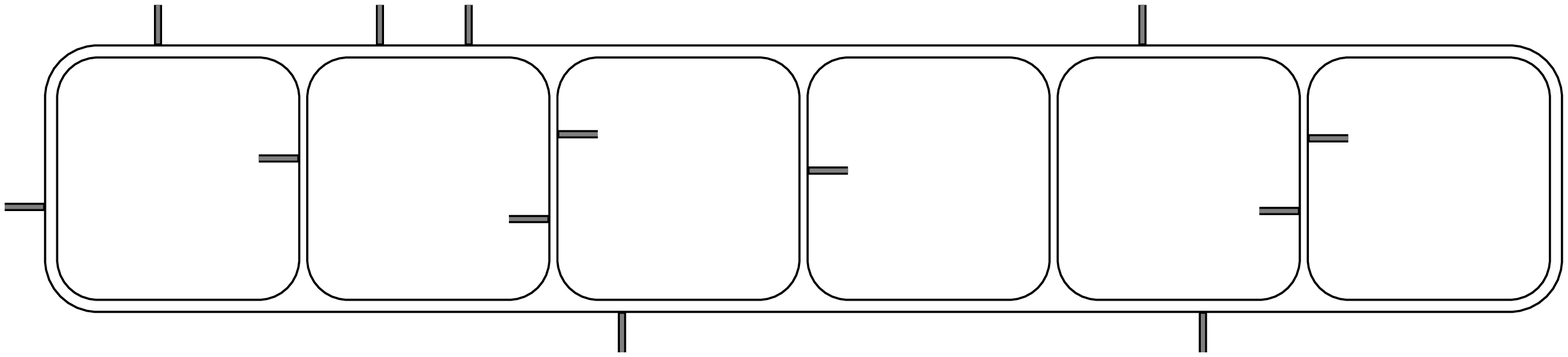}
\caption{Cancelling a tagged loop $\nu$ of $\Gamma''$ by using flowers plus at most one
loop of odd total length.}
\label{cancellation}
\end{center}
\end{figure}

The construction of $\nu^{-1}$ from flowers plus at most one loop of odd total length,
cancelling a tagged loop $\nu$ in $\Gamma''$, is indicated in Figure~\ref{cancellation}.
Each of the small loops in the figure has length of order $40L$, and they are matched along
segments roughly of order $10L$. Adjusting the length of the segments along which
adjacent flowers are paired gives sufficient flexibility to build $\nu^{-1}$ (modulo the
parity issue, which is addressed above). Notice that if $\nu$ contains a long string of tags, each 
distance $\sim 4L$ from the next, we might need to attach two flowers near the midpoint
between two adjacent tags, so that there might be some edges of length $2L$ in the trivalent
partial fatgraph produced at this step. This is good enough to satisfy the conclusion of the
theorem (with some room to spare).

Since $|\Gamma''| = O(\epsilon \cdot|\Gamma|)$ whereas the number of flowers of each kind in
the reservoir is of order $\delta\cdot|\Gamma|$, if we take $\epsilon \ll \delta$ we can
glue up all of $\Gamma''$ this way, at the cost of slightly adjusting the proportion of
each kind of tagged loop in the reservoir.

\subsubsection{Gluing up the reservoir}

We are now left with an almost equidistributed collection of tagged loops of length $40L$ in the
reservoir. Adding to the reservoir the contribution from each $\gamma$ in $\Gamma$, and using
the fact that $\Gamma$ was homologically trivial, we see that the content of the reservoir is
also homologically trivial. It remains to show that any such collection can be glued up to
build a trivalent partial fatgraph with all edges of length at least $L$.

In fact, we only need two kinds of gluings to achieve this: gluings that result in partial
fatgraphs that fatten to {\em annuli} and to {\em pants}. The argument is purely combinatorial,
but quite intricate and involved, and makes up the content of \S~\ref{section:good_pants}.

\begin{remark}\label{Markov_chain_remark}
At this point it is worth spelling out the modifications that need to be made to generalize
the Thin Fatgraph Theorem to random chains generated by an ergodic stationary Markov process of
full support, as discussed in the introduction. 

First, the definition of pseudorandom must be modified. Let's suppose that in our Markov
model, the expected number of copies of a word $\sigma$ in any sufficiently long string $\tau$ is
$E(\sigma)|\tau|$ for some positive $E(\sigma)$. The correct definition of $(T,\epsilon)$-pseudorandomness
of some word $\Gamma$ in this context is that for any cyclic conjugate expressed in the form
$$\Gamma:=w_1w_2\cdots w_Nv$$
with each $w_i$ of length $T$ and at most one word $v$ of length $<T$, for every reduced word
$\sigma$ of length $T$ in $F_k$ there is an estimate
$$1-\epsilon \le \frac {\# \lbrace i\text{ such that } w_i=\sigma\rbrace} N \cdot E(\sigma)^{-1} \le 1+\epsilon$$
Such pseudorandomness holds (with very high probability) for sufficiently long
random words produced by the Markov process.

If one further assumes that $E(\sigma)=E(\sigma^{-1})$ for every $\sigma$, all steps of the
argument above go through (the equality $E(\sigma) = E(\sigma^{-1})$ is used to ensure that after
the random cancellation step the mass of the remainder is small compared to that of the reservoir)
and we are left with a reservoir of loops, where the relative proportion of loops of kind
$\sigma$ and $\sigma'$ is very close to $E(\sigma)/E(\sigma')$. Tagged loops with inverse labels
$\sigma$ and $\sigma^{-1}$ for which the tags are not ``too close'' (under the orientation-reversing
identification of $\sigma$ with $\sigma^{-1}$) can be paired, and therefore we can reduce to the
case of an almost equidistributed collection of tagged loops, at the cost of adjusting the constants.

If one does not assume that $E(\sigma)=E(\sigma^{-1})$ for every $\sigma$, the analogue of the
Thin Fatgraph Theorem is {\em not} true on the nose. But for applications to the construction of
surface subgroups by the method of \S~\ref{section:one_relator_group} it is sufficient to apply
the theorem to (subchains of) chains of the form $r \cup r^{-1}$ where $r$ is a random relator;
now the distribution of $\sigma$ subwords in long segments of $r$ very closely matches the distribution
of $\sigma^{-1}$ subwords in long segments of $r^{-1}$, and the construction goes through.
\end{remark}

\section{Annulus moves and pants moves}\label{section:good_pants}

In this section we show that an almost equidistributed collection of tagged loops of length
$40L$ can be glued up to a trivalent partial fatgraph with all edges of length at least $L$.
Together with the content of \S~\ref{subsection:thin_fatgraph}, 
this will conclude the proof of Theorem~\ref{theorem:thin_fatgraph}.
The technical detail in this section is only necessary because we 
insist that our fatgraphs (surfaces) be orientable.  There is a shortcut 
if we are willing to accept a nonorientable surface, explained in 
Section~\ref{subsection:nonorientable}.

\begin{remark}
For the entirety of this section, we will rescale $40L$ to $L$.  That is, we 
prove that an almost equidistributed collection of tagged loops of length $L$, 
where $L$ is divisible by $4$, can be glued up to a trivalent partial fatgraph 
with all edges of length at least $L/4$.  This rescaling is intended to 
remove meaningless factors of $40$ throughout the argument.
\end{remark}

\subsection{Pants and annuli}

Let $S(L)$ be the set of tagged loops of length $L$, 
where $L$ is divisible by $4$.  Let $W(L)$ be the 
vector space over $\Q$ spanned by $S(L)$; that is, 
$W(L) = \Q[S(L)]$.  We define $h:W(L) \to \Z^k$ to be the 
linear map so that $h(v)$ is the homology class of $v$.  Finally, 
$V(L) = \ker h \subseteq W(L)$ is the vector space of homological trivial 
vectors.  We are interested only in $V(L)$, not $W(L)$, so by 
``full dimensional'', we mean a full dimensional subset of $V(L)$. 
When we say that a vector \emph{projectively bounds} a fatgraph, 
we mean that there is some multiple of the 
vector which has integer coordinates, and the collection of loops represented by the 
integral vector bounds a fatgraph.  A (necessarily integral)
vector \emph{bounds} a fatgraph if the collection of loops that it represents 
bounds a fatgraph.  The uniform vector 
of all $1$'s will be of particular interest, and we denote it by $\1$.

We say that a fatgraph $Y$ with boundary a collection of loops in $S(L)$ 
is \emph{thin} if $Y$ is trivalent and the trivalent vertices 
of $Y$ are pairwise distance 
at least $L/4$ apart, where the tags are counted as trivalent vertices.
Let $C(L)$ be the subset of $V(L)$ of positive vectors which projectively bound a thin
fatgraph.  If $v,w \in C(L)$, then the disjoint union of the thin 
fatgraphs for $v$ and $w$ gives a thin fatgraph for $v+w$.  Also, the definition 
of $C(L)$ shows it to be closed under scalar
multiplication.  Hence, $C(L)$ is a cone.  A variant of the 
\texttt{scallop} \cite{scallop} algorithm gives an explicit hyperplane 
description of $C(L)$, and shows that it is a finite sided polyhedral cone, but we won't 
need this fact in the sequel.

We will build thin fatgraphs out of two kinds of pieces: (good pairs of) \emph{pants} 
and (good) \emph{annuli} (the terminology is supposed to suggest an affinity with
the Kahn--Markovic proof of the Ehrenpreis conjecture, but one should not make too much of this).
A good pair of pants is one
whose edge lengths are all exactly $L/2$ and 
whose tags are each on different edges and exactly distance $L/4$ from 
the real trivalent vertices.  Note the boundary of each such pair 
of pants lies in $V(L)$.  A good annulus is a fatgraph annulus with boundary in $S(L)$ 
whose tags are distance at least $L/4$ apart.  Hereafter, all pants and annuli are good.

Define an involution $\iota:S(L) \to S(L)$ which takes each loop 
to its inverse with the tag moved to the diametrically 
opposite position.  There are several options for the tag at each position -- 
for the definition of $\iota$, we arbitrarily choose any pairing of the options 
to obtain an involution.  There is a special class of annuli, which we 
call $\iota$-annuli, which have boundary of the form $s + \iota(s)$.  
Notice that the collection of all $\iota$-annuli is a thin 
fatgraph which bounds the uniform vector $\1$.

The bulk of our upcoming work lies in manipulating \emph{un}tagged loops, 
and our result here is independently interesting, so 
we will need some complementary definitions.  Let $S'(L)$ be the 
set of untagged loops of length $L$, let $W'(L) = \Q[S'(L)]$, and 
let $V'(L)$ be the vector space of homologically trivial vectors in $W'(L)$.  
We define a thin fatgraph and the uniform vector $\1' \in V'(L)$ 
as before.  The set $C'(L) \subseteq V'(L)$ is the cone of vectors in $V'(L)$ 
which projectively bound thin fatgraphs.  
An untagged good pair of pants is a trivalent 
pair of pants whose edge lengths are exactly $L/2$, and an 
untagged annulus is simply an annulus whose boundary is two 
loops of length $L$.  For untagged loops, $\iota:S' \to S'$ 
is simply inversion, and all annuli are $\iota$-annuli, 
although we may refer to them explicitly as $\iota$-annuli 
to emphasize their purpose.

For many applications, the property of a collection of loops that it
{\em projectively} bounds a thin fatgraph is 
good enough (see e.g.\/ \cite{Calegari_Walker_LP}), and this is in many
ways a more pleasant property to work with, since the set 
of vectors (representing collections of loops)
which projectively bound a thin fatgraph is a cone, whereas the set of
vectors that {\em bound} (i.e.\/ without resorting to taking multiples) is the intersection
of this cone with an integer lattice. However, in this paper it is important to distinguish
between ``bounding'' and ``projectively bounding'', and therefore
in the following propositions, we give both the stronger, technical ``integral''
statement and the weaker, cleaner ``rational'' one.

\begin{proposition}\label{prop:untagged_good_pants}
For any integral vector $v \in V'(L)$, there is $n \in \N$ so that 
$(L/2)v + n\1'$ bounds a collection of good pants and annuli.
Consequently, $C'(L)$ is full dimensional and contains an open 
projective neighborhood of $\1'$.
\end{proposition}

There is a stronger version without the $L/2$ factor if the free group has rank $2$.

\begin{proposition}
\label{prop:rank_2_untagged}
If the free group has rank $2$, then for any integral vector 
$v \in V'(L)$, there is $n \in \N$ so that 
$v + n\1'$ bounds a collection of good pants and annuli.
\end{proposition}

We believe that Proposition~\ref{prop:rank_2_untagged} is probably 
true for higher rank, but the proof would be more complicated than 
we wish for a detail that we do not need.

We delay the rather tedious proof of 
Proposition~\ref{prop:untagged_good_pants} in favor of stating 
the tagged version, which is a corollary and is the version we need.

\begin{proposition}
\label{prop:good_pants}
For any integral vector $v \in V(L)$, there is $n \in \N$ so that 
$(L/2)v + n\1$ bounds a collection of good pants and annuli.
Consequently, $C(L)$ is full dimensional and contains an open 
projective neighborhood of $\1$.
\end{proposition}
\begin{proof}
Let us be given an integral vector $v \in V(L)$.  
Define $f:V(L) \to V'(L)$ to be the map $S(L) \to S'(L)$ 
which forgets the tag, extended by linearity.  
By Proposition~\ref{prop:untagged_good_pants}, we can find a 
collection of pants and annuli which has boundary $(L/2)f(v) + m\1'$. 
Call this fatgraph $Y'$.  Now place 
arbitrary tags in the forced positions on the pants (in the middle of the edges) 
and in allowed positions on the annuli (at least $L/4$ apart) to obtain $Y$.  
Clearly, $Y$ is a thin fatgraph, and $Y$ almost has boundary 
$(L/2)v + m\1$, as desired, but the tags are in the wrong places.  We will fix 
this by simply adding annuli which ``twist'' the tags into the right positions.

Given some pair of pants in $Y$ with boundary 
$\alpha_1 + \alpha_2 + \alpha_3$, let us focus on ``twisting'' the tag 
on $\alpha_1$.  The loop $\alpha_1$ corresponds to a loop 
$\gamma_1$ in $(L/2)v$, and the only difference is that the tag on $\alpha_1$ is 
in a different position from the tag on $\gamma_1$.  There are two cases.  
If the tags on $\iota(\alpha_1)$ and $\gamma_1$ are at least $L/4$ apart, then 
we simply add an annulus with boundary $\iota(\alpha_1) + \gamma_1$.  
If the tags are closer than $L/4$, then we need two annuli: 
one with boundary $\iota(\alpha_1) + \delta_1$, and another with boundary 
$\iota(\delta_1) + \gamma_1$, where here $\delta_1$ is the same loop 
as $\alpha_1$ and $\gamma_1$ with the tag shifted so that the tags on $\delta_1$ 
and $\iota(\alpha_1)$ are at least $L/4$ apart, and similarly for $\iota(\delta_1)$ and $\gamma_1$.  The result is that we 
have added annuli to resolve the boundary from $\alpha_1$ to either 
$\alpha_1 + \iota(\alpha_1) + \gamma_1$ or 
$\alpha_1 + \iota(\alpha_1) + \delta_1 + \iota(\delta_1) + \gamma_1$; that is, 
$\iota$-pairs plus the desired tagged loop $\gamma_1$.
Figure~\ref{fig:resolve_boundary} shows this operation.  

\begin{figure}[ht]
\begin{center}
\labellist
\small\hair 2pt
 \pinlabel {$\alpha_1$} at -1 24
 \pinlabel {$\alpha_2$} at 21 24
 \pinlabel {$\alpha3$} at 42 24
 \pinlabel {$\iota(\alpha_1)$} at 123 11
 \pinlabel {$\delta_1$} at 85 43
 \pinlabel {$\gamma_1$} at 174 34
 \pinlabel {$\iota(\delta_1)$} at 180 3
\endlabellist
\includegraphics[scale=1.5]{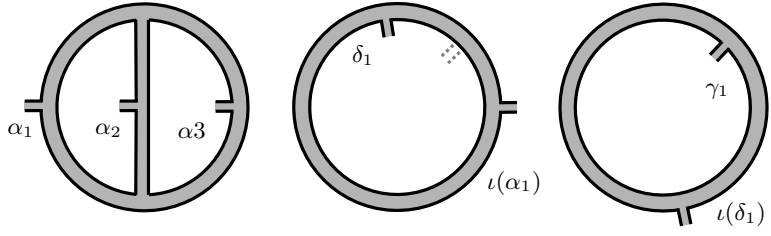}
\caption{Adding annuli to fix an incorrectly tagged boundary.  
We want the loop $\gamma_1$; we have the loop $\alpha_1$, with the 
tag in the wrong place.  The dotted gray tag indicates the proximity of 
the tags on $\iota(\alpha_1)$ and $\gamma_1$.  If it were farther 
away, we'd be done; however, it's too close, so we introduce 
$\delta_1$ so that the distance between the tags on $\iota(\alpha_1)$ and $\delta_1$ 
is at least $L/4$, and the distance between the tags on $\iota(\delta_1)$ and 
$\gamma_1$ is also at least $L/4$.  The result is 
the desired boundary $\gamma_1$, plus the $\iota$  pairs 
$\alpha_1 + \iota(\alpha_1) + \delta_1 + \iota(\delta_1)$.  
This procedure is repeated for $\alpha_2$ and $\alpha_3$.}
\label{fig:resolve_boundary}
\end{center}
\end{figure}

After twisting all the tags in this manner, we are left with 
a collection of pants and annuli with boundary 
$(L/2)v + \Delta + \iota(\Delta) + A + \iota(A)$, where 
$\Delta$ and $A$ are the $\alpha_i$ and $\delta_i$ loops 
used to twist the tags.  By adding $\iota$-annuli, we can make the 
boundary of $Y$ be $(L/2)v + n\1$ for some $n \ge m$, as desired.
\end{proof}

It remains to prove Proposition~\ref{prop:untagged_good_pants}, which 
we now do.

\subsection{Proof of Proposition~\ref{prop:untagged_good_pants}}

Let us be given an integral vector $v' \in V'(L)$. 
The vector $v'$ represents a collection of 
loops $s' \in S'(L)$, for which we must find a thin fatgraph $Y$ of the desired form.  
Our goal is to build a fatgraph which bounds $s' + t + \iota(t)$ for 
some $t \in S'(L)$.  Adding $\iota$-annuli will then immediately 
finish the construction.

If we have a collection of annuli and pants which has boundary 
$s' + t$, then the problem reduces to finding 
a collection of annuli and pants which has boundary of the form
$\iota(t) + u + \iota(u)$ for some $u \in S'(L)$.
We'll repeatedly apply this idea to simplify the problem by 
\emph{attaching pants}.  If we want to have boundary which contains 
a loop $\gamma$, and we find a pair of pants with 
boundary $\gamma + \alpha + \alpha'$, then now we need only 
find boundary containing $\iota(\alpha) + \iota(\alpha')$.  
In this case, we'll say that $\gamma$ and $\iota(\alpha) + \iota(\alpha')$ 
are \emph{pants equivalent}.

For this entire section, we will assume that our free group has rank $2$ and 
is generated by $a$ and $b$.  In \S~\ref{subsection:higher_rank} we explain the extra details
required to deal with free groups of higher rank.

Our strategy will be to start with $s'$ and attach (many) pairs 
of pants which put all the loops in $s'$ into a nice form.  Then we 
attach more pants to further simplify the loops, and so on, 
eventually reducing to a case that is simple enough to handle by hand.
A \emph{run} in a loop is a maximal subword of the form $a^p$ or 
$b^p$ for some integer power $p \ne 0$.  Note that any loop contains 
an even number of runs.  We first reduce to the case that every loop
has at most $4$ runs, then to the case that every loop has $4$ runs in a 
nice arrangement, then $2$ runs, which we address directly.

For clarity, we separate the simplification into lemmas. 

\begin{lemma}
\label{lem:pants_to_4_runs}
Any loop is pants-equivalent to a collection of loops with at most $4$ runs.
\end{lemma}
\begin{proof}
To begin, we show how to attach a pair of pants to a loop which 
produces two loops, each of which has fewer runs than the initial 
loop.  This method works whenever the number of runs is more than $4$, 
so it reduces the loops in $s'$ to a collection of loops 
with at most $4$ runs.  Let us be given a loop $\gamma$.  
The easiest way to visualize attaching a pair of pants is to simply 
draw a diameter $d$ on $\gamma$ between two antipodal vertices.  
Labeling the diameter $d$ produces a pair of pants attached to $\gamma$.  
Note that we must be careful to label $d$ compatibly 
with the labels adjacent to the vertices to which we attach $d$, 
so that the vertices do not fold. 
See Figure~\ref{fig:attach_diameter}.

\begin{figure}[ht]
\begin{center}
\includegraphics[scale=1.5]{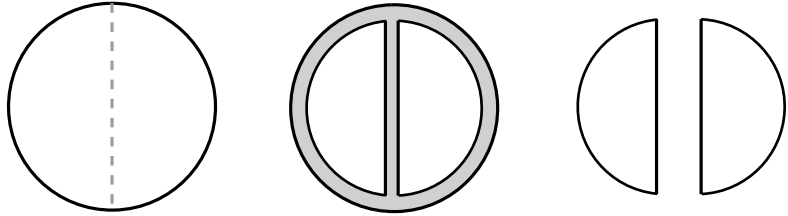}
\caption{Attaching a diameter to a loop to form a pair of pants.  
The pair of loops on the right is pants equivalent to the loop on the left.}
\label{fig:attach_diameter}
\end{center}
\end{figure}

For concreteness, let us number the vertices of $\gamma$ by 
$0$, \dots, $L-1$, and we let $d_i$ be the (oriented) 
diameter with initial vertex $i$ (and thus terminal vertex
$(i+L/2)\bmod L$).  Each diameter $d_i$ divides $\gamma$ into 
two pieces.  Let $x_i$ be the number of runs in the non-cyclic 
subword of $\gamma$ starting at index $i$ and of length $L/2$; 
that is, the number of runs in the word to the ``right'' of $d_i$.  
Similarly, let $y_i$ be the number of runs to the left.  
See Figure~\ref{fig:change_by_one}.  Let $r$ be the number 
of runs in $\gamma$.  Note that $x_i+y_i$ may be greater than $r$.  
Specifically, $r \le x_i + y_i \le r+2$.  The important 
feature of these numbers is that $|x_i - x_{i+1}| \le 1$ and 
$|y_i- y_{i+1}| \le 1$.  This is easily seen by considering 
the combinatorial possibilities that occur as we rotate the starting 
point $i$ around the loop $\gamma$. 
We are particularly interested in \emph{matched runs}, which are 
runs separated in either direction by the same number of other runs.  
That is, matched runs are ``directly across'' from one another 
in the list of runs (we use scare quotes to emphasize that matched runs are not 
antipodal in the same sense that ``antipodal vertices'' are).

\begin{figure}[ht]
\begin{center}
\includegraphics[scale=0.8]{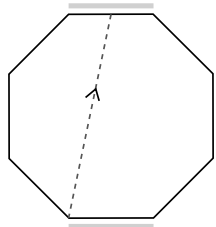}
\caption{Moving the diameter one position changes $x_i$ and $y_i$ 
by at most $1$.  For the marked diameter, we have $x_i = 5$ and $y_i=4$.
A pair of matched runs is marked in grey.}
\label{fig:change_by_one}
\end{center}
\end{figure}

The ``functions'' $x_i$ and $y_i$ can be interpolated to piecewise linear
functions on the circle; by applying the intermediate value theorem to these
interpolations, we deduce that there is some point at which 
the interpolated graphs intersect.  This can happen
either at some value of $i$, in which case $x_i = y_i$, 
or between two values $i$ and $i+1$, but 
by the discussion above, in this latter case 
$|x_i - y_i| \le 1$ and $|x_{i+1}- y_{i+1}|\le 1$.  
In either case, $d_i$ must intersect two matched runs $R$ and $R'$, 
perhaps on the boundaries of the runs.  Now decrease $i$ until one of the 
ends of $d_i$ lies on the boundary of $R$ or $R'$.  This puts 
$d_i$ in to one of two combinatorial configurations, up to 
rotation and symmetry.  See 
Figure~\ref{fig:matched_runs}.  Note that the configuration 
on the right cannot occur at the intersection point, since $|x_i-y_i|=2$. 

\begin{figure}[ht]
\begin{center}
\includegraphics[scale=0.8]{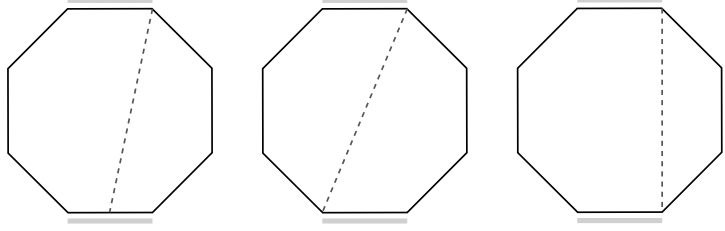}
\caption{Possible configurations of $d_i$ with respect to the matched 
runs $R$ and $R'$.  Up to rotation and symmetry, there are two.  
Note the configuration on the right cannot occur.}
\label{fig:matched_runs}
\end{center}
\end{figure}

We handle the two cases separately.  First, the more generic case illustrated 
in Figure~\ref{fig:matched_runs} on the left.  Here we label the diameter 
entirely with the generator which is \emph{not} the one labeling the 
bottom run, and in such a way as to minimize the number of runs in the 
resulting two loops.  
Figure~\ref{fig:cut_to_four}, left, illustrates this.  
The sign of the labels on the diameter depends on the 
signs and orders of the generators around the endpoints, but the 
picture is equivalent.  Notice that the 
number of runs in each of the resulting loops is at most $r/2 + 2$.  

\begin{figure}[ht]
\begin{center}
\labellist
\small
\pinlabel {$a$}  at 47 -5
 \pinlabel {$b$}  at 86 11
 \pinlabel {$a$}  at 103 47
 \pinlabel {$b$}  at 89 86
 \pinlabel {$a$}  at 49 101
 \pinlabel {$b$}  at 9 86
 \pinlabel {$a$}  at -4 51
 \pinlabel {$b$}  at 10 13
 \pinlabel {$A$}  at 56 9
 \pinlabel {$B$}  at 75 21
 \pinlabel {$A$}  at 88 46
 \pinlabel {$B$}  at 78 74
 \pinlabel {$A$}  at 48 87
 \pinlabel {$B$}  at 22 76
 \pinlabel {$A$}  at 10 50
 \pinlabel {$B$}  at 19 23
 \pinlabel {$A$}  at 36 9
 \pinlabel {$b$}  at 47 48
 \pinlabel {$B$}  at 61 43
 
 \pinlabel {$a$}  at 192 -4
 \pinlabel {$b$}  at 227 11
 \pinlabel {$a$}  at 245 48
 \pinlabel {$b$}  at 232 87
 \pinlabel {$a$}  at 193 103
 \pinlabel {$b$}  at 154 88
 \pinlabel {$a$}  at 140 50
 \pinlabel {$b$}  at 153 15
 \pinlabel {$A$}  at 192 9
 \pinlabel {$B$}  at 220 21
 \pinlabel {$A$}  at 233 48
 \pinlabel {$B$}  at 221 76
 \pinlabel {$A$}  at 191 88
 \pinlabel {$B$}  at 166 77
 \pinlabel {$A$}  at 154 49
 \pinlabel {$B$}  at 163 24
 \pinlabel {$a$}  at 179 35
 \pinlabel {$A$}  at 191 27
 \pinlabel {$b$}  at 193 69
 \pinlabel {$B$}  at 207 63
\endlabellist
\includegraphics{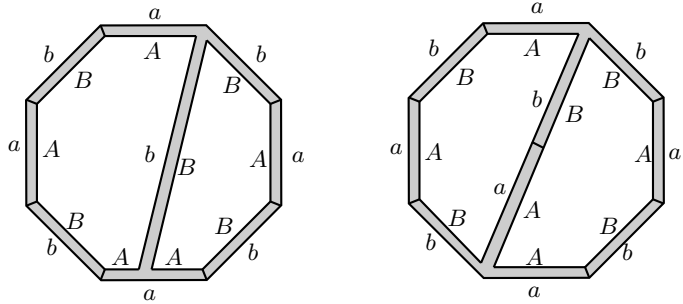}
\caption{Labeling the diameter to reduce the number of runs. Each label 
represents a potentially long run of that generator.}
\label{fig:cut_to_four}
\end{center}
\end{figure}

In the non-generic case illustrated Figure~\ref{fig:matched_runs} in the middle, 
we label half of the diameter with one generator and the other half with 
the other, in a way which is compatible with the top and bottom labels.  
See Figure~\ref{fig:cut_to_four}, right.
In certain cases, it is possible to label the entire diameter with 
a single generator, and this reduces the number of runs still further, but 
we have illustrated the worst situation.  We therefore compute again that the 
number of runs in each of the resulting loops is $r/2 + 2$.

As long as $r/2 + 2 < r$, or $r>4$, this will produce two loops of 
strictly smaller length.  Repeatedly attaching pants resolves our collection 
$s'$ into a new collection of loops, all of which have at most $4$ runs.
\end{proof}

We have shown that an arbitrary collection of loops is pants 
equivalent to a collection of loops with at most $4$ runs.  In order 
to further reduce this to $2$ runs, we first need to make the 
$4$-run loops balanced.  A $4$-run loop is \emph{balanced}
if there is a pair of diameters $d_a$ and $d_b$ at right angles 
(with all endpoints spaced exactly $L/4$ apart)
such that $d_a$ starts touching one $a$ run and ends touching 
the other, and similarly for $d_b$ with the $b$ runs.  Here 
\emph{touching} a run means that the vertex on which the diameter 
starts or ends lies between two letters, at least one of which lies 
within the run.  See Figure~\ref{fig:balanced_runs}.

\begin{figure}[ht]
\begin{center}
\includegraphics[scale=0.8]{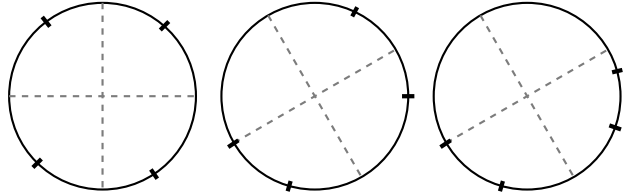}
\caption{Examples of two balanced loops (left) and an unbalanced loop (right)}
\label{fig:balanced_runs}
\end{center}
\end{figure}

\begin{lemma}
\label{lem:4_runs_to_balanced}
Any $4$-run loop is pants equivalent to a collection of balanced loops 
and $2$-run loops.
\end{lemma}
\begin{proof}
Suppose we are given a $4$-run loop.  Without loss of 
generality, let us suppose there are at least as many $a$'s as $b$'s, and 
let $G = \#a - \#b$ be the \emph{generator inequity}, 
recording how many more $a$'s than $b$'s there are.
Let $x$ and $x'$ be the number of $b$'s in the longer and 
shorter $b$ runs, respectively.  Abusing notation, 
we'll also refer to the runs themselves as $x$ and $x'$.  Note that 
the $a$ or $b$ runs may have  negative exponents, so they are actually 
runs of $A$ or $B$.  For simplicity, we'll use the ``positive'' notation.

First, let us eliminate the case that there are very few $b$'s.  
Suppose that $x + 2x' < L/2$.  Consider the two diameters starting at the ends of $x$.  
There are two cases: if $x=x'$ and 
the runs are exactly antipodal, then drop a diameter between the middles of $x$ and 
$x'$; the diameter at right angles will touch 
both $a$ runs, and the loop will be balanced, as desired.
Otherwise, one of the diameters misses $x'$ entirely, and we 
can cut to produce a $2$-run loop and a loop whose generator 
inequity is strictly smaller (the roles of $a$ and $b$ are 
reversed, and the inequity becomes $2x' < 2(L/2 - x - x') = G$, i.e. 
smaller than the current inequity).  
See Figure~\ref{fig:reduce_generator_inequity}.

\begin{figure}[ht]
\labellist
\small
 \pinlabel {$a$}  at 56 -3
 \pinlabel {$b$}  at 87 69
 \pinlabel {$a$}  at 45 93
 \pinlabel {$b$}  at 6 76
 
 \pinlabel {$a$}  at 191 1
 \pinlabel {$b$}  at 210 71
 \pinlabel {$a$}  at 171 92
 \pinlabel {$b$}  at 129 74
 \pinlabel {$a$}  at 121 37
 \pinlabel {$A$}  at 186 15
 \pinlabel {$B$}  at 197 62
 \pinlabel {$A$}  at 171 78
 \pinlabel {$B$}  at 142 67
 \pinlabel {$A$}  at 134 42
 \pinlabel {$b$}  at 162 52
 \pinlabel {$B$}  at 172 43
\endlabellist
\includegraphics[scale=1]{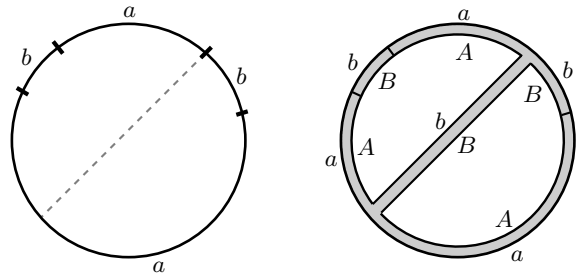}
\caption{If the $b$ runs are short, we can cut to reduce the 
generator inequity.}
\label{fig:reduce_generator_inequity}
\end{figure}

After repeatedly reducing the generator inequity, 
we may assume that $x + 2x' \ge L/2$. We conclude that 
$x \ge L/6$ and $x' \ge L/4 - x/2$, so $L/3 \le x+ x'$, and $G \le L/3$.
Our loop is now less degenerate, but it still might not be balanced.
We'd like for $x$ to be positioned opposite 
$x'$, as shown in Figure~\ref{fig:balanced_runs}, left, 
so that we can simply draw two diameters and be done.  However, 
the loop might look like Figure~\ref{fig:balanced_runs}, right, 
in which $x$ and $x'$ are too close.  
In order to remedy this, we introduce the 
\emph{triangle move at $x$} technique.

Algebraically, a triangle move at $x$ takes in a word $a^{e_1}b^{x}a^{e_2}b^{x'}$ 
(for clarity, assume all the exponents are positive) such that 
$x+e_2 > L/2$ and $x < L/2$ and builds a pair of pants with boundary
\[
a^{e_1}b^{x}a^{e_2}b^{x'} + A^{e_1 + L/2-x}B^{x}A^{e_2-(L/2-x)}B^{x'} 
 + a^{L/2-x}b^{x}A^{L/2-x}B^{x}
\]
The notation obscures the function of a triangle move, which is shown 
in Figure~\ref{fig:triangle_move} and is as follows:
it produces one balanced loop (opposite runs have the same 
length), and another 
loop with the same run sizes as the original one, except that 
$L/2-x$ of the $a$'s in the top run have been shifted down 
to the bottom run.  The signs of the generators 
may change as they shift, but we are only concerned with the lengths 
at this point.  This is the critical feature of the triangle 
moves --- shifting generators from one run to the other without disturbing 
anything else.

\begin{figure}[ht]
\labellist
\small
\pinlabel {$a$}  at 44 93
 \pinlabel {$b$}  at 95 125
 \pinlabel {$a$}  at 50 158
 \pinlabel {$b$}  at -2 126
 
 \pinlabel {$a$}  at 152 92
 \pinlabel {$b$}  at 205 124
 \pinlabel {$a$}  at 160 158
 \pinlabel {$b$}  at 105 127
 
 \pinlabel {$A$}  at 150 106
 \pinlabel {$B$}  at 190 136
 \pinlabel {$A$}  at 175 144
 \pinlabel {$A$}  at 127 144
 \pinlabel {$B$}  at 119 126
 
 \pinlabel {$A$}  at 167 117
 \pinlabel {$a$}  at 177 128
 \pinlabel {$B$}  at 153 128
 \pinlabel {$b$}  at 162 141
 
 \pinlabel {$A$}  at 258 98
 \pinlabel {$A$}  at 293 114
 \pinlabel {$B$}  at 265 141
 \pinlabel {$A$}  at 235 155
 \pinlabel {$B$}  at 218 128
 
 \pinlabel {$a$}  at 315 116
 \pinlabel {$B$}  at 335 132
 \pinlabel {$A$}  at 300 160
 \pinlabel {$b$}  at 284 146

 \pinlabel {$x'$} at 65 21
 \pinlabel {$x$} at 122 1
 \pinlabel {$x'$} at 175 20
 \pinlabel {$x$} at 234 69

 \endlabellist
\includegraphics[scale=1]{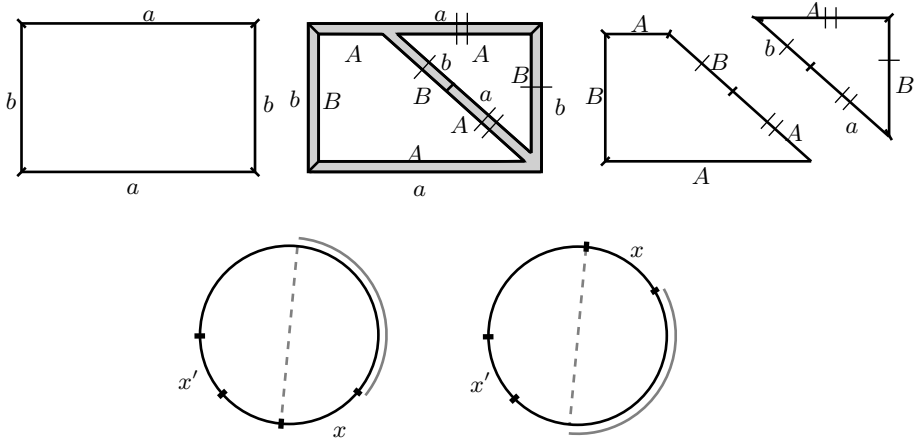}
\caption{A triangle move shown explicitly, top, and 
schematically, bottom.  The hash marks in the top picture 
indicate segments of equal length.  Note the output 
is a balanced loop (opposite runs have exactly the same length), 
plus a loop in which $L/2-x$ 
of the $a$'s on the top run have been shifted down.  The schematic 
picture shows the effective result: a triangle move at $x$ 
on the left loop produces balanced loop (not shown), plus the 
loop on the right; note the runs have been shifted so that 
the right loop is now balanced.}
\label{fig:triangle_move}
\end{figure}

In the Figure~\ref{fig:triangle_move} schematic, we are able to 
perform a triangle move at $x$ in such a way that the loop becomes balanced.  
We will show that this can always be done.  We need two things 
to happen simultaneously: $x$ and $x'$ must be opposite enough 
so that there is a diameter between them, and the diameter at 
right angles must also touch the $a$ runs.

At this point, we split the argument into two cases.  First, assume 
that $x \le L/4$.  In this case, $x$ and $x'$ are short enough 
that if we can find a diameter which touches both $x$ and $x'$, then 
the diameter at right angles automatically touches the $a$ runs, so 
the loop will be balanced.  
Consider the run $x'$: it casts a ``shadow'' directly 
opposite it so that if any part of $x$ touches the shadow, we 
succeed in placing a diameter between the runs.  Looking at 
the initial endpoint of $x$, we must place this endpoint 
in the \emph{target region}, which we define to be the segment 
of size $x+x'$ ending directly antipodal to the endpoint of $x'$. 
See Figure~\ref{fig:balancing_loop_small_x}.

\begin{figure}[ht]
\labellist
\small\hair 2pt
 \pinlabel {$x'$} at 7 8
 \pinlabel {$x$} at 87 17
\endlabellist
\includegraphics[scale=1]{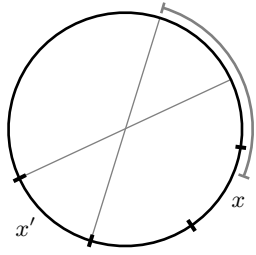}
\caption{The run $x'$ casts a shadow  directly opposite itself.
In the case that $x \le L/4$, if we succeed in shifting $x$ so it 
touches this shadow, the loop
will be balanced.  In other words, we need to place the initial 
endpoint of $x$ inside the target region, shown in gray on the outside
of the loop.}
\label{fig:balancing_loop_small_x}
\end{figure}

Let $t$ be the size of the target region, so $t = x + x'$.
Doing a triangle move at $x$ moves the initial endpoint of $x$ 
by the \emph{shift size}, which we denote $s$.  Recall that 
$s = L/2 - x$.  If we can show that $s \le t$, then obviously we can 
shift $x$ until it lies within the target region.  Also recall 
that we reduced the generator inequity, so we have $x \ge L/6$ and 
$x+x' \ge L/3$.  Putting these together, we have 
\[
s = \frac{L}{2} - x \le \frac{L}{3} \le x + x' = t.
\]
Therefore we do indeed have $s \le t$ and we can balance the loop.
This finishes the case that $x \le L/4$.

For the other case, assume $x \ge L/4$.  We will use the 
same technique, shifting $x$ until the loop is balanced.
Here we must be careful: if $x \ge L/4$ it is no longer obvious that there exists 
a diameter at right angles which exhibits the loop as balanced, 
so we must take this into account when setting the size of the 
target region.  In this case, the target region starts 
exactly $L/4$ after the initial endpoint of $x'$ and ends 
exactly $L/4+x$ before the final endpoint of $x'$.  
Figure~\ref{fig:balancing_loop_big_x} shows the target region.

\begin{figure}[ht]
\labellist
\small\hair 2pt
 \pinlabel {$x'$} at -2 32
 \pinlabel {$x$} at 38 85
\endlabellist
\includegraphics[scale=1]{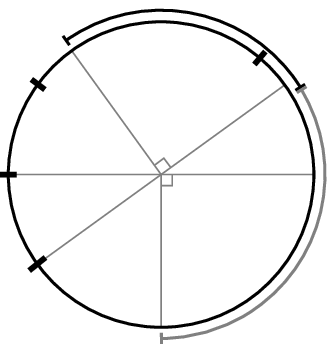}
\caption{In the case that $x \ge L/4$, we need to be 
careful about the size of the target region (shown in gray).  
The initial endpoint of $x$ can be placed anywhere in the target 
region, which starts 
exactly $L/4$ after the initial endpoint of $x'$ and ends 
exactly $L/4+x$ before the final endpoint of $x'$.  A segment 
of length $x$ is shown adjacent to the target region 
for illustrative purposes.}
\label{fig:balancing_loop_big_x}
\end{figure}

Computing the size, we find in this case that $t = L/2 + x' - x$.  
Again, the shift size is $s = L/2 -x$, so we immediately get
\[
s = \frac{L}{2}-x < \frac{L}{2}-x + x' = t.
\]
This completes the proof in the case that $x \ge L/4$ and thus 
the proof of the lemma.
\end{proof}

At this point, we are left with a collection of balanced $4$-run 
loops.  For reducing the loop, even nicer than balanced is a loop 
with a \emph{diameter between corners}.  A diameter between 
corners is a diameter which starts between an $a$ and 
$b$ run and ends between the other $a$ and other $b$ runs.
See Figure~\ref{fig:diameter_between_corners}.

\begin{figure}[ht]
\includegraphics[scale=1]{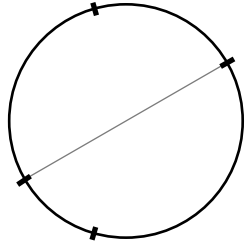}
\caption{This loop has a diameter between corners.}
\label{fig:diameter_between_corners}
\end{figure}

\begin{lemma}
\label{lem:balanced_to_diameter between corners}
Any balanced loop is pants equivalent to a collection of $2$-run loops and 
loops with a diameter between corners.
\end{lemma}
\begin{proof}
The proof of this lemma is essentially contained in the moves 
shown in Figure~\ref{fig:quartering}.  Given a balanced loop, 
there exist diameters at right angles with ends in opposite runs.
First, attach a pair of pants along the diameter between the $a$ runs 
by labeling the diameter entirely with $b$'s.  If one of the 
ends of this diameter touches a $b$ run, it is necessary to 
be careful about the labels to ensure there is no folding.  
If both ends touch $b$ runs, the loop already has a diameter between 
corners and we are done.

\begin{figure}[ht]
\labellist
\small
 \pinlabel {$a$}  at 44 85
 \pinlabel {$b$}  at 94 134
 \pinlabel {$a$}  at 44 184
 \pinlabel {$b$}  at -3 136
 
 \pinlabel {$a$}  at 132 91
 \pinlabel {$a$}  at 170 90
 \pinlabel {$b$}  at 203 135
 \pinlabel {$a$}  at 175 180
 \pinlabel {$a$}  at 137 182
 \pinlabel {$b$}  at 105 147
 \pinlabel {$A$}  at 165 105
 \pinlabel {$B$}  at 186 135
 \pinlabel {$A$}  at 168 167
 \pinlabel {$A$}  at 138 168
 \pinlabel {$B$}  at 117 138
 \pinlabel {$A$}  at 138 104
 \pinlabel {$b$}  at 144 136
 \pinlabel {$B$}  at 160 137
 
 \pinlabel {$A$}  at 253 95
 \pinlabel {$b$}  at 274 136
 \pinlabel {$A$}  at 254 178
 \pinlabel {$B$}  at 224 138
 
 \pinlabel {$A$}  at 327 95
 \pinlabel {$b$}  at 357 109
 \pinlabel {$b$}  at 357 162
 \pinlabel {$A$}  at 325 178
 \pinlabel {$B$}  at 301 157
 \pinlabel {$B$}  at 299 118
 \pinlabel {$a$}  at 331 106
 \pinlabel {$B$}  at 345 118
 \pinlabel {$B$}  at 343 157
 \pinlabel {$a$}  at 330 166
 \pinlabel {$b$}  at 314 152
 \pinlabel {$b$}  at 313 122
 \pinlabel {$a$}  at 330 131
 \pinlabel {$A$}  at 329 143
 
 \pinlabel {$a$}  at 42 -4
 \pinlabel {$B$}  at 95 28
 \pinlabel {$a$}  at 51 64
 \pinlabel {$b$}  at -3 31
 
 \pinlabel {$a$}  at 168 -2
 \pinlabel {$B$}  at 218 36
 \pinlabel {$a$}  at 186 67
 \pinlabel {$a$}  at 145 66
 \pinlabel {$b$}  at 119 33
 \pinlabel {$A$}  at 166 13
 \pinlabel {$b$}  at 205 35
 \pinlabel {$A$}  at 192 50
 \pinlabel {$A$}  at 146 50
 \pinlabel {$B$}  at 133 32
 \pinlabel {$A$}  at 189 18
 \pinlabel {$a$}  at 197 29
 \pinlabel {$B$}  at 168 39
 \pinlabel {$b$}  at 180 49

 \pinlabel {$A$}  at 275 4
 \pinlabel {$A$}  at 312 26
 \pinlabel {$B$}  at 288 49
 \pinlabel {$A$}  at 256 63
 \pinlabel {$B$}  at 235 34
 
 \pinlabel {$a$}  at 335 27
 \pinlabel {$b$}  at 357 39
 \pinlabel {$A$}  at 332 67
 \pinlabel {$b$}  at 310 49
 
 \pinlabel {$L/4$}  at 81 31
 \pinlabel {$L/2$}  at 45 50

\endlabellist
\includegraphics[scale=0.9]{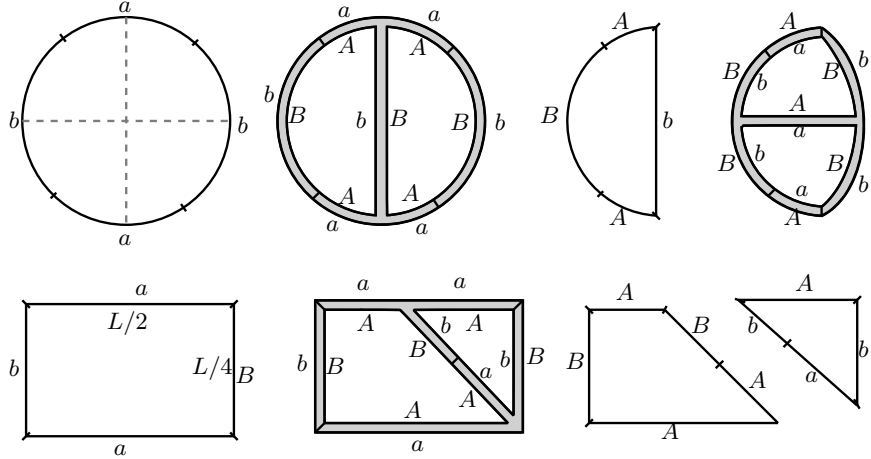}
\caption{Applying the moves described in Lemma~\ref{lem:balanced_to_diameter between corners} 
to reduce a balanced loop to a loop with a diameter between corners
Lengths are not to scale, and some lengths are labeled.  
The operation is read left to right, 
top to bottom.}
\label{fig:quartering}
\end{figure}

The result of attaching this pair of pants is two new loops.  They have the 
same pattern, so we'll focus on one of them.  It must be of the form 
(assuming positive exponents) $b^{L/2}a^{e_1}b^{e_2}a^{e_3}$, and furthermore, 
because the original loop was balanced, the diameter starting exactly in the 
middle of the $b^{L/2}$ run must touch the other $b$ run.  Attach a pair of 
pants by labeling this diameter entirely with $a$.  
If the diameter touches one of the $a$ runs, 
as always, we must label it to avoid folding.  This results in two new loops, 
which either have two runs, or have the form $b^{L/4}a^{L/2}b^{e_1}a^{e_2}$ 
(again assuming positive exponents and new exponent variables $e_1$ and $e_2$).

Again these two loops have the same pattern, so we focus on one of them.  
The final step is to do a triangle move at the $b^{L/4}$ run.  This shifts 
exactly $L/4$ of the $a$'s between runs, and results in 
a loop of the form $b^{L/4}a^{L/4}b^{e_1}a^{e_2+L/4}$.  Observe this has a diameter 
between corners.  The triangle move also produces a byproduct loop; 
we did not stress this earlier because we did not need it, 
but the byproduct loop has opposite runs of exactly the same length, so 
it has a diameter between corners.  This completes the proof.
\end{proof}

\begin{lemma}
\label{lem:diameter_between_corners_to_2_runs}
Any loop with a diameter between corners is pants equivalent to 
a collection of $2$-run loops.
\end{lemma}
\begin{proof}
This step requires at most two triangle moves, and again, is essentially 
described by a picture, which is shown in Figure~\ref{fig:perfecting}.

\begin{figure}[ht]
\labellist
\small
\pinlabel $a$ at 55 100
\pinlabel $b$ at 86 145
\pinlabel $a$ at 34 176
\pinlabel $b$ at 6 134

\pinlabel $a$ at 161 105
\pinlabel $b$ at 186 161
\pinlabel $a$ at 150 179
\pinlabel $b$ at 111 131
\pinlabel $A$ at 152 119
\pinlabel $B$ at 171 160
\pinlabel $A$ at 159 165
\pinlabel $B$ at 127 132
\pinlabel $A$ at 150 130
\pinlabel $a$ at 159 144
\pinlabel $B$ at 131 143
\pinlabel $b$ at 138 157

\pinlabel $B$ at 215 132
\pinlabel $A$ at 274 99

\pinlabel $a$ at 293 128
\pinlabel $b$ at 250 153
\pinlabel $A$ at 264 176
\pinlabel $B$ at 314 146

 \pinlabel {$a$}  at 65 3
 \pinlabel {$b$}  at 98 45
 \pinlabel {$A$}  at 45 77
 \pinlabel {$b$}  at 17 30
 
 \pinlabel {$a$}  at 170 6
 \pinlabel {$b$}  at 200 49
 \pinlabel {$A$}  at 150 81                                                                                             
 \pinlabel {$b$}  at 122 33                                                                                             
 \pinlabel {$A$}  at 159 20                                                                                             
 \pinlabel {$B$}  at 183 51                                                                                             
 \pinlabel {$a$}  at 155 66                                                                                             
 \pinlabel {$B$}  at 136 36                                                                                             
 \pinlabel {$A$}  at 153 37                                                                                             
 \pinlabel {$a$}  at 164 49
 
 \pinlabel {$A$}  at 277 0                                                                                              
 \pinlabel {$A$}  at 265 35                                                                                             
 \pinlabel {$B$}  at 228 24                                                                                             
 \pinlabel {$a$}  at 287 40                                                                                            
 \pinlabel {$B$}  at 322 53                                                                                             
 \pinlabel {$a$}  at 285 77
\endlabellist
\includegraphics[scale=1]{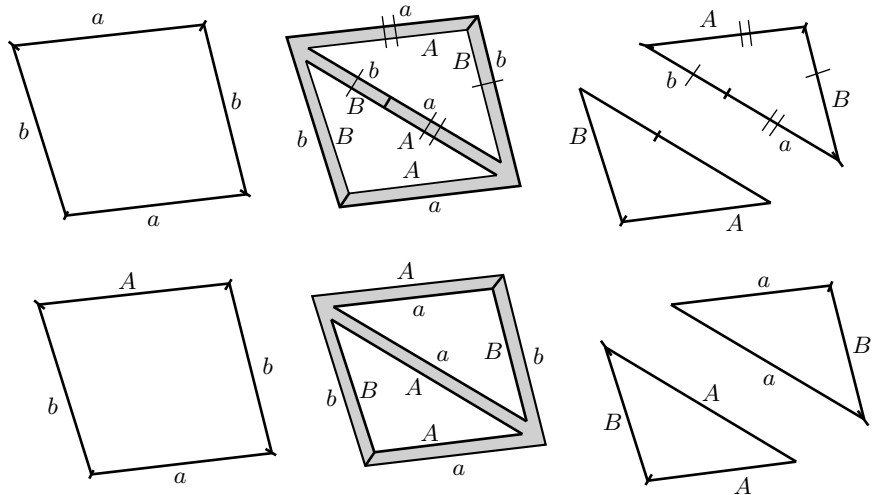}
\caption{Performing triangle moves on loops with diameters between corners 
in order to produce $2$-run loops.  There are two possibilities, depending 
on whether an inverse pair appears.}
\label{fig:perfecting}
\end{figure}

Let us assume without loss of generality that our loop is 
of the form $b^{e_1}a^{e_2}b^{e_3}a^{e_4}$, where 
$|e_1| + |e_2| = |e_3| + |e_4| = L/2$.  At this stage, we need to 
differentiate between positive and negative exponents.  
Suppose that $e_2$ and $e_4$ (the $a$ runs) have the same sign.  
Then doing a 
triangle move at the run $b^{e_1}$ produces a $2$-run loop and 
the loop $a^{e_2}b^{\pm e_1}a^{-e_2}b^{-e_1}$. 
When following Figure~\ref{fig:perfecting}, one must remember that 
the loop is on the inside of the pants, so the orientation is backwards.

Therefore, $b^{e_1}a^{e_2}b^{e_3}a^{e_4}$ is pants 
equivalent to (a $2$-run loop and) the inverse of this new loop, 
i.e. $b^{e_1}a^{e_2}b^{\mp e_1}a^{-e_2}$. Observe that this 
loop still has a diameter between corners, but now the signs on the $a$ 
runs are different.  

We have reduced to a loop of the form
$b^{e_1}a^{e_2}b^{e_3}a^{e_4}$, where 
$|e_1| + |e_2| = |e_3| + |e_4| = L/2$, and where the signs of 
$e_2$ and $e_4$ are different.  In this case, we may attach a 
pair of pants by labelling the diameter entirely by $a$'s.  This produces 
two $2$-run loops, and completes the proof.
\end{proof}

We are now left with a collection of only $2$-run loops.
For the next step, we will modify our collection of loops with $2$ runs 
to put them in a standard form.  A \emph{uniform} loop has a single run, so just 
one generator appears.  An \emph{even} loop has two runs of the same length 
(so length $L/2$).  The \emph{type} of a loop with two runs is a pair 
that records which generators appear, so for example $(a,B)$.

\begin{lemma}
\label{lem:2_runs_to_nice}
Any collection of loops with $2$ runs is pants equivalent to a collection 
of uniform loops, even loops, and at most one loop of each type.
\end{lemma}
\begin{proof}
This step involves shifting and combining loops into even and uniform loops, 
which will leave a finite remainder.  All of these operators 
are performed on loops of the same type.  The first step is to 
arbitrarily select a generator to be the \emph{small} generator on each loop.  
We'll choose $b$.  Any loop whose $b$ run has length over $L/2$ can 
be cut with a diameter labeled with just $a$ to produce an even loop 
and a loop with $b$ run length less than $L/2$, as shown in 
Figure~\ref{fig:make_b_small}.  
\begin{figure}[ht]
\labellist
\small
 \pinlabel {$b$} at 76 7
 \pinlabel {$a$} at 30 91
 \pinlabel {$b$} at 197 15
 \pinlabel {$B$} at 186 29
 \pinlabel {$a$} at 163 42
 \pinlabel {$A$} at 149 54
 \pinlabel {$B$} at 122 48
 \pinlabel {$b$} at 108 52
 \pinlabel {$A$} at 149 79
 \pinlabel {$a$} at 136 89
 \pinlabel {$b$} at 324 11
 \pinlabel {$A$} at 283 45
 \pinlabel {$a$} at 261 49
 \pinlabel {$b$} at 216 45
\endlabellist
\includegraphics[scale=1]{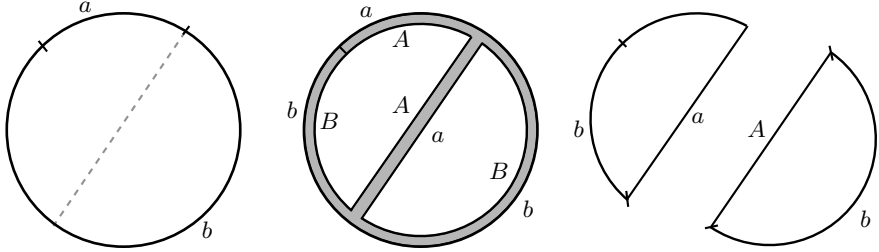}
\caption{If a loop has a $b$ run length larger than $L/2$, 
it can be cut to produce an even loop and a loop with a 
$b$ run shorter than $L/2$.  See the text for a discussion 
of why the inner boundary components become inverse on the right.}
\label{fig:make_b_small}
\end{figure}
An important feature of cutting to reduce the size of the $b$ run 
is that it doesn't change the type of the loop.  We have been 
somewhat casual about this thus far, because it makes the 
pictures easier to understand, but recall from the introduction 
to this section that if we have a loop $\gamma$, and we find a 
pair of pants with boundary $\gamma + \alpha + \alpha'$, then 
the remaining problem is to find a collection of pants 
with boundary $\iota(\alpha) + \iota(\alpha')$; that is, recall that 
$\gamma$ is not pants equivalent to $\alpha + \alpha'$, but rather 
is pants equivalent to $\iota(\alpha) + \iota(\alpha')$.  Therefore, 
as shown in Figure~\ref{fig:make_b_small}, the input is the
loop of type $(a,b)$, and the output is an even loop and another 
loop of type $(a,b)$.  The same holds true for the other 
operations we describe here.

We aren't concerned 
with the even loops, so we turn our attention to the loops with $b$ 
run length strictly less than $L/2$.  There are two necessary operations 
here; the \emph{trade}, in which we swap pieces of the 
$b$ run between loops, and the \emph{combine}, in which we combine 
two small $b$ loops into a single one (and produce a uniform loop 
and two even loops as byproducts).  
The combine operation works 
on any two loops whose total number of $b$'s is strictly less than $L/2$.
These operations are shown in 
Figures~\ref{fig:trading} and~\ref{fig:combining}.  
Algebraically, the trade operation takes in $a^{p_1}b^{t_1 + t_2}$ and 
$a^{p_2}b^{w_1 + w_2}$, where $t_1, t_2, w_2>0$ and $w_1 \ge0$, and produces 
$a^{p_1}b^{t_1 + w_2}$ and $a^{p_2}b^{w_1 + t_2}$.  The combine 
operator takes in $a^{p_1}b^{r_1}$ and $a^{p_2}b^{r_2}$, where $r_1 + r_2 < L/2$, 
and produces $a^{p_3}b^{r_1+r_2}$.

\begin{figure}[ht]
\labellist
\small\hair 2pt
 \pinlabel {$b$} at 38 174
 \pinlabel {$a$} at 42 74

 \pinlabel {$b$} at 99 94
 \pinlabel {$a$} at 133 3
 
 \pinlabel {$b$} at 185 110
 \pinlabel {$a$} at 206 146
 \pinlabel {$b$} at 188 178
 \pinlabel {$a$} at 141 136
 
 \pinlabel {$a$} at 287 101
 \pinlabel {$b$} at 280 173
 \pinlabel {$A$} at 256 159
 \pinlabel {$B$} at 235 120
 
 \pinlabel {$a$} at 168 16
 \pinlabel {$b$} at 220 29
 \pinlabel {$a$} at 205 69
 \pinlabel {$b$} at 186 90

 \pinlabel {$a$} at 291 6
 \pinlabel {$b$} at 296 77
 \pinlabel {$A$} at 265 67
 \pinlabel {$B$} at 249 31
 
 \pinlabel {$a$} at 354 75
 \pinlabel {$b$} at 362 176

 \pinlabel {$a$} at 382 9
 \pinlabel {$b$} at 440 93
\endlabellist

\includegraphics[scale=0.75]{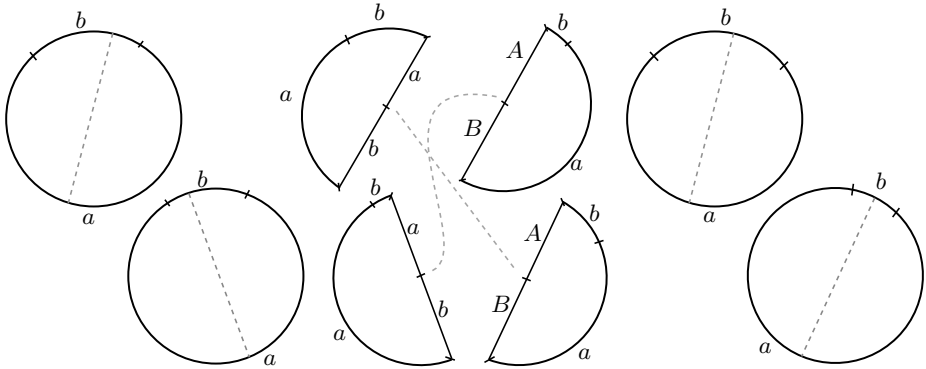}
\caption{Trading pieces of $b$ runs between loops}
\label{fig:trading}
\end{figure}

\begin{figure}[ht]
\labellist
\small\hair 2pt
 \pinlabel {$b$} at 19 196
 \pinlabel {$a$} at 78 114
 \pinlabel {$B$} at 50 148
 \pinlabel {$b$} at 42 162
 
 \pinlabel {$b$} at 95 39
 \pinlabel {$a$} at 10 83
 \pinlabel {$b$} at 40 54
 \pinlabel {$B$} at 52 40
 
 \pinlabel {$b$} at 201 81
 \pinlabel {$b$} at 126 140
 \pinlabel {$a$} at 121 59
 \pinlabel {$a$} at 185 139
 \pinlabel {$b$} at 148 96
 \pinlabel {$B$} at 158 108
 
 \pinlabel {$b$} at 251 55
 \pinlabel {$a$} at 299 142
 \pinlabel {$a$}  at 275 93
 \pinlabel {$A$} at 275 108
 
 \pinlabel {$b$} at 343 134
 \pinlabel {$a$} at 425 66
\endlabellist
\includegraphics[scale=0.75]{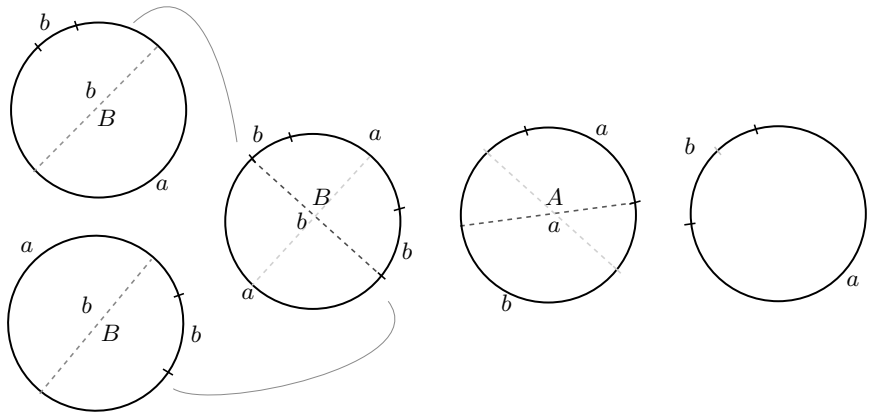}
\caption{Combining $b$ runs onto a single loop.  The first step 
produces a byproduct uniform $a$ loop, the second an even loop, and 
the third another even loop.}
\label{fig:combining}
\end{figure}

Notice that the trade operation requires only that \emph{one} of the 
diameters be interior in the $b$ run, or as written above, $w_2 > 0$ but $w_1 \ge 0$.  
Using the trade operation, 
we can take every $b$ run and, for each one, move all but a single $b$ onto a single 
chosen loop.  Whenever this chosen loop contains a
$b$ run longer than $L/2$, we cut it off and start trading again.  
After this, we are left with a single loop with an unknown length $b$ run, and 
possibly many loops with a single $b$.  Then, we use the combine operation 
to combine these loops into a smaller number of loops with longer $b$ runs.  
Then we trade the $b$ mass to our chosen loop again, and so on.

After trading and combining as much as possible, we are left with 
either no loop (if the only remaining loop is even), 
a single loop with a $b$ run of length less than $L/2$, or two loops 
whose combined run length is exactly $L/2$.  This last case arises because 
we cannot use the combine operation on these loops.  There is yet another 
sequence of moves, however, to resolve this: we use the 
trade operation to obtain two loops with $b$ runs of length exactly $L/4$.  
Then we cut and join them to produce a perfectly balanced loop with $4$ runs.  
A single triangle move results in an even loop plus a commutator; we then
apply Figure~\ref{fig:perfecting} to the commutator to get a pair of inverse loops. 

Doing this to each loop type proves the lemma.
\end{proof}

The final step in the proof of Proposition~\ref{prop:untagged_good_pants} 
is to show that we can attach pants and annuli to the output of 
Lemma~\ref{lem:2_runs_to_nice}, that is, uniform loops, even loops, 
and a single remainder loop of each type, so that we have nothing left.  
Let $x_{a,b}$, $x_{a,B}$, $x_{A,b}$, and $x_{A,B}$ denote the run 
length of the single $b$ run in the remainder loop of each type.  
Rescaling and considering arbitrary $L$, we can think of each 
variable as a real number in the interval $[0,1/2)$.  The fact that 
the entire collection of loops must be homologically trivial 
gives us two linear equations counting the homology contributions 
to $a$ and $b$, respectively:
\begin{align*}
x_{a,b} + x_{A,b} - x_{a,B} - x_{A,B} &= k_1 \\
(1-x_{a,b}) + (1-x_{a,B}) - (1-x_{A,b}) - (1-x_{A,B}) &= k_2.
\end{align*}
Since there are even and uniform loops to consider, it is not 
{\it a priori} the case that $k_1 = k_2 = 0$.  However, it \emph{is} 
the case that $k_1, k_2 \in \frac{1}{2}\Z$, and $k_1 \pm k_2 \in \Z$, 
since the uniform and even loops change homology discretely by $1$ and $1/2$, 
respectively.  Now consider $k_1 + k_2 = 2(x_{A,b} - x_{a,B})$.  
Since $0 \le x_{A,b},x_{a,B} < 1$, we have $k_1+k_2 \in (-1,1)$, so 
$k_1 + k_2 = 0$.  A similar argument shows that $k_1 - k_2 = 0$, so 
$k_1 = k_2 = 0$.  Therefore, $x_{a,b} = x_{A,B}$ and $x_{a,B} = x_{A,b}$.
These equalities show that actually, the remainder loops we 
have must be inverse pairs, so they are the boundary of two annuli.
The homologically trivial collection 
of uniform and even loops can now be glued along entire runs, so is 
obviously pants equivalent to the empty collection.  

That is, we are originally given a collection $s'$ of loops of length $L$, 
and after applying Lemmas~\ref{lem:pants_to_4_runs}, 
\ref{lem:4_runs_to_balanced}, \ref{lem:balanced_to_diameter between corners}, 
\ref{lem:diameter_between_corners_to_2_runs}, \ref{lem:2_runs_to_nice}, 
and the argument above, we have shown that our 
original collection was pants equivalent to the 
empty collection, meaning have successfully produced a collection of pants and 
annuli with boundary $s' + t + \iota(t)$, where $t$ is the many 
intermediate boundaries we used to reduce 
$s'$.  Adding in $\iota$-annuli, then, gives us a 
collection of pants and annuli which has boundary $s' + n\1'$, 
for some sufficiently large $n$.

Observe that we never need to duplicate our collection $s'$, or, 
equivalently, multiply $v'$ by any factor.  We have therefore proved 
the stronger Proposition~\ref{prop:rank_2_untagged} for rank $2$ 
free groups.

\subsection{Higher rank}\label{subsection:higher_rank}

We have completed the proof of Proposition~\ref{prop:untagged_good_pants} 
in the case that the free group has rank $2$.  We now describe the 
necessary modifications to the argument for higher rank free groups.
Given a collection of loops $s' \in S'(L)$, the same technique 
of cutting with diameters works to show that $s'$ is pants 
equivalent to a collection of $4$-run loops.  However, 
triangle moves no longer apply, since each of the $4$ runs might be 
a run of a different generator.

The first step is to attach pants in such a way that we 
are left with $4$-run loops, each of which only involves two generators.  
This is actually quite straightforward, since we have more freedom 
with the labels on the diameters that we attach.  
Figure~\ref{fig:reduce_to_rank_2} shows how to attach 
diameters to produce loops of the desired form which are 
pants equivalent to the original loop.  Technically, the figures 
represent simply unions of pants, not the (non-trivalent) fatgraphs shown.  
They are drawn as shown to emphasize the point that we produce several 
other byproduct loops, but they come in cancelling inverse pairs.  
All the interior diameters shown have length $L/2$, even though they are not 
drawn to scale.

\begin{figure}[ht]
\labellist
\small\hair 2pt
 \pinlabel {$a$} at 67 -5
 \pinlabel {$b$} at 146 59
 \pinlabel {$c$} at 58 120
 \pinlabel {$d$} at -5 63
 \pinlabel {$A$} at 78 12
 \pinlabel {$B$} at 127 73
 \pinlabel {$C$} at 107 103
 \pinlabel {$B$} at 110 80
 \pinlabel {$C$} at 36 100
 \pinlabel {$D$} at 17 63
 \pinlabel {$D$} at 33 46
 \pinlabel {$A$} at 64 16
 \pinlabel {$D$} at 89 20
 \pinlabel {$b$} at 101 66
 \pinlabel {$d$} at 85 58
 \pinlabel {$D$} at 51 78
 \pinlabel {$a$} at 44 23
 \pinlabel {$d$} at 48 35
 
 \pinlabel {$a$} at 258 -5
 \pinlabel {$b$} at 326 60
 \pinlabel {$c$} at 251 121
 \pinlabel {$b$} at 176 58
 \pinlabel {$A$} at 273 13
 \pinlabel {$B$} at 311 70
 \pinlabel {$C$} at 286 107
 \pinlabel {$C$} at 215 99
 \pinlabel {$B$} at 197 69
 \pinlabel {$b$} at 282 30
 \pinlabel {$B$} at 283 65
 \pinlabel {$b$} at 242 86
 \pinlabel {$C$} at 224 59
 \pinlabel {$c$} at 227 43
 \pinlabel {$a$} at 241 40                                                                                                                                                        
 \pinlabel {$A$} at 253 27                                                                                                                                                        
\endlabellist     
\includegraphics[scale=1.1]{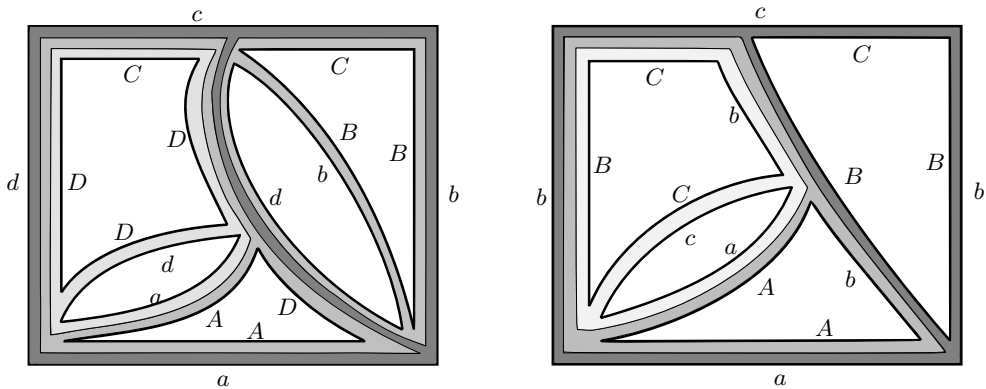}
\caption{Reducing each loop to lie in a rank $2$ subgroup.  The fatgraph on 
the left is built out of folded pants only when $b$ and $d$ are distinct 
generators.  If $b=d^{\pm 1}$, then we use the picture on the right.  A similar 
picture holds when the vertical diameter has endpoints on runs of the same generator.  
The point is that all of the non inverse-matched pants boundaries 
lie in rank $2$ subgroups.}
\label{fig:reduce_to_rank_2}
\end{figure}

We remark that the pictures in Figure~\ref{fig:reduce_to_rank_2} are general, up to 
rotation and reflection.  
The double-diameter from top to bottom exists by the argument in the proof 
of Lemma~\ref{lem:pants_to_4_runs}, and the double-diameter from the lower left 
corner to the middle exists because the first diameter has length $L/2$.
We also remark that in higher rank, it is possible that we have some $3$-run loops.  
It is simple to cut these to $4$-run loops and then apply the above argument.

After applying Lemma~\ref{lem:2_runs_to_nice}, we may assume that we are left entirely
with uniform loops, even loops, and one $2$-run loop of each type.  Now, though, there are 
$4\binom{r}{2}$ loop types, which is too many to duplicate the 
linear-algebraic argument from the previous section.  The simple solution 
is to take $L/2$ copies of our collection.  Now each loop type is 
repeated exactly $L/2$ times, so when we re-collect the remainder, 
we are left with no remainder, so we have only uniform loops and 
even loops, which can be paired arbitrarily.  Therefore, 
we have found a collection of pants and annuli which 
has boundary $(L/2)v' + n\1'$, which completes the proof 
of Proposition~\ref{prop:untagged_good_pants}.

\begin{remark}
The statement of Proposition~\ref{prop:good_pants} 
doesn't quantify how $n$ depends on the size of $v$, but this is technically 
necessary to deduce Theorem~\ref{theorem:thin_fatgraph} from 
Proposition~\ref{prop:good_pants}. Following the steps
of the argument shows directly that $n=O(\|v\|_1)$; however, one can deduce
the existence of such a linear bound on general grounds, as we now indicate.
If we fix $L$, then the cone of homologically trivial collections of 
tagged loops of length $L$ is a finite sided rational cone, so it 
has a {\em finite} Hilbert basis $B$. Applying Proposition~\ref{prop:good_pants}
to each basis vector $b$ gives a constant $n(b)$ so that $(L/2)b + n(b)\1$ bounds
a collection of good pants and annuli. Since every integral $v$ can be expressed
as a disjoint union of copies of basis vectors, we obtain a uniform linear estimate for
$n$.  To deduce Theorem~\ref{theorem:thin_fatgraph}, 
we observe that by making $\epsilon \ll \delta$, we can make the 
distribution of tagged loops as close to identical as desired. 

Note that for random (rather than pseudorandom) words, 
if the mass of the reservoir is $N$, the central limit theorem says that
the deviation from equidistribution will be of order $\sqrt{N}$.
\end{remark}

\subsection{Nonorientable surfaces}\label{subsection:nonorientable}

Proposition~\ref{prop:good_pants} shows that any sufficiently 
uniform collection $s$ of tagged loops can be glued up into a thin fatgraph 
(in fact, can be glued up into just annuli and pants).  
We apply Proposition~\ref{prop:good_pants} as the last step in the 
proof of Theorem~\ref{theorem:thin_fatgraph} to find 
many closed surface subgroups of random groups.  We claim that 
Proposition~\ref{prop:good_pants} is the right way to do this, 
for reasons discussed at the end of this section.  However, if we were only 
interested in the existence of surface subgroups, we can replace 
Proposition~\ref{prop:good_pants} with a trick which avoids most of
the technical difficulty in this section.

Consider a collection $s$ of tagged loops. 
Duplicate the collection once, so we assume that every loop in $s$ 
appears an even number of times.  
Now take an untagged loop $\gamma$.  In the collection $s$, there are 
many tagged copies of $\gamma$, and the tags are almost equidistributed 
around $\gamma$.  Pair up these tagged loops such that in every pair, the tags are 
almost antipodal, and certainly distance $L/4$ apart.  
Let's consider a single pair $\gamma_1$ and $\gamma_2$.  
They are both tagged copies of the loop $\gamma$.  There's no annulus 
with {\em oriented} boundary $\gamma_1 + \gamma_2$, but there {\em is} an annulus
bounding $\gamma_1+\gamma_2$ in such a way that the orientation of one
of the $\gamma_i$ {\em disagrees} with that it inherits from the annulus.
Performing this pairing for all pairs gives a collection of annuli bounding all the
tagged $\gamma$-loops. 

This construction will give rise to {\em nonorientable} surfaces in random groups,
which will be certified to be $\pi_1$-injective in the sequel. Taking an index 2 subgroup
gives an oriented surface subgroup.

There are at least two good reasons to justify the hard work that went into the proof
of Proposition~\ref{prop:good_pants}. The first is that this proposition is of independent
interest, and can be used as a combinatorial tool in many contexts where it is important
to build orientable surfaces or fatgraphs. The second is that the nonorientable surface
subgroups built using the trick above will never be essential in $H_2$, whereas the surfaces
built using the full power of the Thin Fatgraph Theorem {\em can} be taken to be
homologically essential in random groups whenever the density $D$ is positive.

We remark that the original Kahn--Markovic construction of surface subgroups in
hyperbolic 3-manifolds necessarily produced nonorientable (and therefore homologically
inessential) quasifuchsian surfaces. Very recently, Liu--Markovic \cite{Liu_Markovic}
have shown how to modify
the construction --- using substantial ingredients from the Kahn--Markovic proof of the
Ehrenpreis conjecture --- to build orientable quasifuchsian surfaces (projectively) realizing
any homology class.

\section{Random one-relator groups}\label{section:one_relator_group}

Throughout this section we fix a free group $F_k$ with $k\ge 2$ generators, and we fix
a free generating set. For some big (unspecified) constant $n$, we let $r$ be a random 
cyclically reduced word of length $n$, and we consider the one-relator group 
$G:=\langle F_k\; | \; r \rangle$.

In this section we will show that with probability going to 1 as $n \to \infty$, the group $G$
contains a surface subgroup $\pi_1(S)$. The surface $S$ in question can be built from $N$ disks bounded by $r$
and $N$ disks bounded by $r^{-1}$, glued up along their boundary in such a way that the $1$-skeleton is a trivalent
fatgraph with every edge of length $\ge L$, where $L$ is some (arbitrarily big) constant fixed in
advance, and $N \le 20L$ is the constant in the Thin Fatgraph Theorem~\ref{theorem:thin_fatgraph}.

The group $G$ is evidently $C'(\lambda)$ for any $\lambda>0$ with probability going to 1 as
$n \to \infty$, and therefore the injectivity of $S$ can be verified by showing that the $1$-skeleton
of $S$ does not contain a long path in common with $r$ or $r^{-1}$, except for a path contained
in the boundary of one of the $2N$ disks.

A trivalent graph $Y$ in which every edge has length at least $L$ has at most $2|Y|\cdot 2^{m/L}$
subpaths of length $m$, where $|Y|$ is the length of $Y$. In our context the trivalent graph $Y$
will arise as the 1-skeleton of a surface $S$ constructed as above, so $|Y|=O(n)$.
In a free group of rank $k$ there are
(approximately) $(2k-1)^m$ reduced words of length $m$, and the relator $r$ contains at most $2n$ of
them (allowing inverses). 
So if we fix any positive $\epsilon'$, and take $m=\epsilon'\cdot n$, then providing we choose $L$
so that $1/L \log_{2k-1} 2 < \epsilon'$
a simple counting argument shows that $Y$ does not have any
path of length $m$ in common with an {\em independent} random word of length $n$, with probability
$1-O(e^{-n^c})$. However, $Y$ and $r$ are utterly dependent, and we must work harder to show 
that $S$ is injective. 

The key idea is the observation that disjoint subwords of a long random relator $r$ are
(almost) {\em independent of each other}. Informally, we fix some small positive $\delta$, and
break up $Nr \cup Nr^{-1}$ into pieces (called {\em beads}) of size $n^{1-\delta}$ which each bound 
their own trivalent fatgraph (by the Thin Fatgraph Theorem). 
Then subpaths in the fatgraph associated to one bead {\em will} 
be independent of the subpaths of $Nr \cup Nr^{-1}$ associated to another bead, and this argument
can be made to work.

\begin{remark}
The ``simple counting argument'' alluded to above is a special case of
Gromov's intersection formula (\cite{Gromov_asymptotic} \S~9.A) 
which implies that two sets of independent random words of a fixed length whose (multiplicative) densities
sum to less than 1 are typically disjoint. We apply this observation in a more substantial way
in \S~\ref{section:random_groups}, especially in the proof of Theorem~\ref{theorem:surface_random_group}.
\end{remark}

\subsection{Independence and correlation}\label{subsection:independence}

Since this is the first point in the argument where we are using the genuine randomness of the
relators (rather than just pseudorandomness), some remarks are in order.

A random word in a finite alphabet (in the uniform distribution) has the property that any two
disjoint subwords are independent. A random (cyclically) {\em reduced} word in the free group fails to have
this property, since (for example) if $uv$ are adjacent subwords, the last letter of $u$
must not cancel the first letter of $v$ (so the words are not really independent). However, such
words have a slightly weaker property which is just as useful as independence in most circumstances;
this property can be summarized by saying that {\em correlations decay exponentially}. 

To explain the meaning of this, let's fix reduced words $u$ and $v$ and a distance $T$. Suppose
$r$ is a random (reduced) word, and let's write $r$ as $abcd$ where $|a|=|u|$, where $|b|=T$,
where $|c|=|v|$, and where $|d|=n-|a|-|b|-|c|$. The probability that $c=v$ is $1/(2k)(2k-1)^{v-1}$.
Saying that correlations decay exponentially means that the probability that $c=v$ {\em conditioned
on $a=u$} satisfies
$$1- (2k-2)^{-T} < \frac {\Pr(c=v \;| \; a=u)} {\Pr(c=v)} < 1+ (2k-2)^{-T}$$
This estimate is elementary; 
see e.g.\/ \cite{Calegari_Walker_RR}, Lem.~2.4 for a careful proof. In the more general context
of stationary Markov chains, such exponential bounds on the deviation from independence in the
tail are called {\em Chernoff bounds}.

In the sequel we are typically interested in estimating the probability of finding a pair of
matching subwords which are separated by a considerable distance. Only order-of-magnitude
estimates of the probability are important for our arguments. So in practice we can treat the
subwords as though they were independent.

\subsection{Beads}

Let $r$ be a random cyclically reduced word of length $n$. We fix some (small) positive constants 
$C$ and $\delta$; later we will say how small they should be. 

\begin{definition}
A {\em bead decomposition} of $r$ is a decomposition of $r$ into segments labeled (in cyclic order)
$r_0, r_1^+, r_2^+,\cdots, r_{M-1}^+, r_M, r_{M-1}^-,\cdots, r_1^-$ where each $r_i^\pm$ has 
length $n^{1-\delta}\cdot(1/2+o(1))$ and $r_0, r_M$ have length $n^{1-\delta}\cdot(1+o(1))$
(so that $M=n^{\delta}\cdot(1+o(1))$) so that the prefix of $r_i^+$ of length $C\log(n)$ is inverse to
the suffix of $r_i^-$, and the prefix of $r_M$ of length $C\log(n)$ is inverse to the suffix of $r_M$.
\end{definition}

See Figure~\ref{bead_decomposition}.

\begin{figure}[htpb]
\labellist
\small\hair 2pt
\pinlabel $r_0$ at -10 30
\pinlabel $r_1^+$ at 100 0
\pinlabel $r_1^-$ at 100 60
\pinlabel $r_2^+$ at 175 0
\pinlabel $r_2^-$ at 175 60
\pinlabel $\cdots$ at 288 0
\pinlabel $\cdots$ at 288 60
\pinlabel $r_{M-1}^+$ at 400 0
\pinlabel $r_{M-1}^-$ at 400 60
\pinlabel $r_M$ at 500 30
\endlabellist
\centering
\includegraphics[scale=0.6]{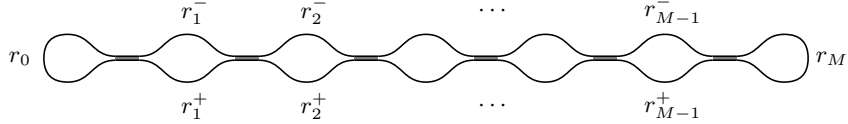}
\caption{Bead decomposition of $r$}\label{bead_decomposition}
\end{figure}

Given a bead decomposition of $r$, we glue the mutually inverse subwords described above, thereby
decomposing $r$ into a sequence of loops (``beads'') of length $n^{1-\delta}\cdot(1+o(1))$ with
one or two tags. Denote these beads $B_0,B_1,\cdots, B_M$. The edges of length $C\log(n)$ obtained
by gluing the inverse prefixes/suffixes we refer to as the {\em lips} of the beads.

\begin{lemma}[Bead Decomposition]\label{lemma:bead_decomposition}
There exists a bead decomposition with probability $1-O(e^{-n^c})$.
\end{lemma}
\begin{proof}
We start at an arbitrary location in $r$, and take this to be the approximate midpoint of $r_0$.
Extend this region $n^{1-\delta}\cdot(1/2 - n^{-\delta})$ in either direction, and then read the
next pair of segments of length $n^{1-2\delta}$ synchronously until the first time we read off a 
pair of mutually inverse segments of length $C\log(n)$. If $C<(1-2\delta)/\log(2k-1)$ then
such a pair of mutually inverse segments will occur with probability $1-O(e^{-n^c})$ by Chernoff's
inequality, for some $c>0$ depending on $\delta$. We then build successive segments $r_i^\pm$
in order by the same procedure. There are $n^{\delta}$ such segments, and this polynomial term
can be absorbed into the exponential estimate of probability 
at the cost of adjusting constants.
\end{proof}

In order to think about the (lack of) correlation between subwords of the different $B_i$, the following
mental picture is useful: imagine that we generate the word $r$ by a 
Markov process letter by letter as we go, starting at the center of $r_0$ and building outwards.
In this model, the letters making up each successive $B_i$ are
only generated after we have already constructed $B_j$ with $j<i$.

There is no reason to expect that the beads $B_i$ are homologically trivial, but there is a trick
to adjust them so that they are. We build a bead decomposition of $r$ and of $r^{-1}$ simultaneously,
so that the beads of $r^{-1}$ have inverse labels to the beads of $r$. We denote the beads of
$r^{-1}$ by $B_i^{-1}$, and note that they are inverse (as tagged cyclic words) to the $B_i$.
Then for each $i$ the union $B_i \cup B_i^{-1}$ is homologically trivial. 

By Theorem~\ref{theorem:thin_fatgraph} for each $i$, and for some $N\le 20L$,
the collection $NB_i \cup NB_i^{-1}$ bounds a trivalent fatgraph
$Y_i$ with all edges of length at least $L$, with probability $1-O(e^{-n^c})$. 
Since we can first build the fatgraph $Y_i$ in a way which depends only on the substrings
$r_i^\pm$, the Chernoff bound says that for any positive $\alpha$ there is $c(\alpha)$ so that with probability
$1-O(e^{-n^c})$ there are no paths in $Y_i$ of length $n^\alpha$ in common with any segment in
$r- r_i^\pm$. Summing over the $n^\delta$ different indices $i$, and absorbing this polynomial factor
into the probability estimate (at the cost of adjusting constants), we see that with probability
$1-O(e^{-n^c})$ there is no index $i$ and no path in any $Y_i$ of length $n^\alpha$ in common with
any segment of $r-r_i^\pm$.

\begin{lemma}[No long path]\label{lemma:no_long_piece}
Let $\beta>0$ be fixed.
Let $Y$ be the fatgraph obtained from the union of the $Y_i$ associated to a bead decomposition
as above. Then there is some positive $c(\beta)$ so that with probability $1-O(e^{-n^c})$ 
every path in $Y$ of length $\beta n$
which appears in $r$ or $r^{-1}$ is in $\partial S(Y)$.
\end{lemma}
\begin{proof}
Let $\gamma$ be a path in $Y$ of length $\beta n$ and let $\gamma' \subset r$ (without
loss of generality) have the
same labels as $\gamma$. Then for any fixed $\alpha>0$ there is a $c>0$ so that
for each $i$ and each subsegment $\sigma'$ of $\gamma'$ of length $n^{\alpha}$ contained in
$B_i$ the corresponding subsegment $\sigma$ of $\gamma$ must have at least $(1-o(1))$ of its
length contained in $Y_i$, with probability $1-O(e^{-n^c})$. By the definition of the bead decomposition,
successive subsegments of $\gamma'$ in adjacent $B_i$ are joined by paths of length $C\log(n)$ running
over the lip. The corresponding subsegments in $\gamma$ that transition from $Y_i$ to $Y_{i+1}$
must {\em also} run over the lip, so there is another copy of the word on the lip contained in $r$
within distance $n^{\alpha}$ of the lip. 
If the two copies are not distinct, so that $\gamma$ and $\gamma'$ overlap on a common
path, then since $Y$ is folded we must simply have $\gamma = \gamma'$ and $\gamma$ is in $\partial S(Y)$
as claimed. Otherwise there are two distinct copies of the lip contained in a segment of length
$n^{\alpha}$ in $r$. See Figure~\ref{coincidence}.

\begin{figure}[htpb]
\centering
\labellist
\small\hair 2pt
 \pinlabel {$\gamma'$} at 62 20
 \pinlabel {$\gamma$} at 61 4
 \pinlabel {$\gamma'$} at 86 80
 \pinlabel {$\gamma$} at 84 141
 \pinlabel {$\tau$} at 275 130
 \pinlabel {$\tau'$} at 278 156
 \pinlabel {$\tau$} at 204 182
 \pinlabel {$\tau'$} at 324 68
 \pinlabel {$\tau$} at 178 28
 \pinlabel {$\tau'$} at 307 28
\endlabellist
\includegraphics[scale=0.64]{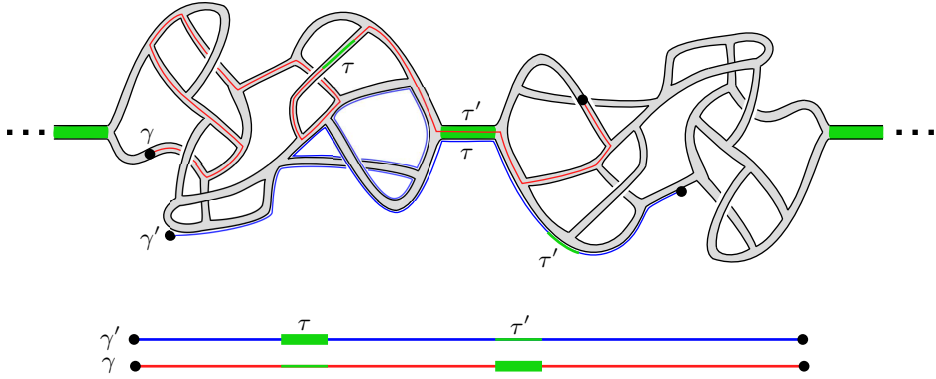}
\caption{The figure shows part of a surface obtained by fattening the spine $Y$, with boundary
contained in part of a single copy of $r$ (for simplicity).
Every time $\gamma' \subset r$ (blue) runs over a (thick green) 
lip $\tau$ (which has length $C\log(n)$), 
the path $\gamma$ (red) must also run over $\tau$ at almost the same time.
Since $\gamma$ and $\gamma'$ have the same labels, the copy of $\tau$ in $\gamma$ 
gives rise to a {\em coincidence}; i.e.\/ another copy 
$\tau'$ (thin green) of $\tau$ in $\gamma'$ within distance $n^\alpha$ of the first copy.}
\label{coincidence}
\end{figure}

If $\alpha$ is sufficiently small compared to $C$, the probability 
that two identical subwords of length $C\log(n)$ will occur in a specific segment of length $n^\alpha$
is arbitrarily small (in fact, of size $O(n^{-\alpha'})$ for some $\alpha'$ depending
on $\alpha$ and $C$). Explicitly, in a segment of length $n^\alpha$, the expected
length of the longest pair of identical subwords is $2\alpha\log_{2k-1}(n)$, so if we choose 
$C> 2\alpha/\log(2k-1)$ we obtain the desired estimate (with $\alpha'$ depending on the difference between
$C$ and $2\alpha/\log(2k-1)$). Remember that $C<(1-2\delta)/\log(2k-1)$ and our only constraint on
$\delta$ is that it is positive. So we can achieve $C>2\alpha/\log(2k-1)$ subject to this constraint if
$\alpha$ is sufficiently small.

If $\gamma'$ has length $\beta n$ it must therefore run over $\beta n^\delta$ successive
lips in this way, and each time it runs over a lip it must contain two identical (or inverse)
subwords of length $C\log(n)$ within distance $n^\alpha$ of each other; call each pair of identical
(or inverse) subwords a {\em coincidence}. Each coincidence occurs with probability $O(n^{-\alpha'})$.
The probability of successive coincidences at adjacent lips is not independent, but the Chernoff bound
says that the probability of $O(n^\delta)$ such coincidences in succession is of order 
$O((n^{-\alpha'})^{n^\delta}) = O(e^{-n^c})$ for some $c$, and the lemma is proved.
\end{proof}

We deduce the following theorem as a corollary:

\begin{theorem}[Random One-Relator]\label{theorem:surface_one_relator}
Fix a free group $F_k$ and let $r$ be a random cyclically reduced word of length $n$. Then
$G:=\langle F_k\; | \; r \rangle$ contains a 
surface subgroup with probability $1-O(e^{-n^c})$.
\end{theorem}
\begin{proof}
This follows from Lemma~\ref{lemma:no_long_piece} and Lemma~\ref{lemma:injective_surface}.
\end{proof}

There are $O(e^{n^c})$ many choices of bead decomposition, and
most of these give rise to quasiconvex surface subgroups.

\begin{definition}\label{definition:bead_surface}
We call the surfaces constructed as above {\em beaded surfaces}.
\end{definition}

Note that a beaded surface has genus $O(n)$. If $k=2$ the least genus of a beaded surface
is $o(n)$, since we can take $N=1$ in the application of the
Thin Fatgraph Theorem, and then as $n \to \infty$ we can take $L \to \infty$. It seems very likely
that the least genus of a beaded surface is $o(n)$ for any fixed $k\ge 2$.

\begin{remark}
It will turn out (after the proof of Theorem~\ref{theorem:surface_random_group}) 
that beaded surfaces are quasiconvex
(in fact, they stay quasiconvex even after adding many more random relators),
but it is more efficient to give the proof of this in the next section.
\end{remark}

\section{Random groups}\label{section:random_groups}

In this section we prove our main theorem, that a random group at density $D<1/2$ contains
a surface subgroup with probability $1-O(e^{-n^c})$. In fact, our argument shows that it contains
{\em many} subgroups (of genus $O(n)$). Our argument
depends on some elements of the theory of small cancellation developed for random
groups by Ollivier \cite{Ollivier_small}, and we refer to that paper several times.

\subsection{Small cancellation in random groups}\label{subsection:Ollivier_results}

For later convenience, we here state three results from Ollivier \cite{Ollivier_small}
that we use in the sequel.

\begin{theorem}[Ollivier, \cite{Ollivier_small}, Thm.~2]\label{theorem:Ollivier_isoperimetric_theorem}
Let $G$ be a random group at density $D$. Then for any positive
$\epsilon$, and any reduced van Kampen diagram
$\D$ containing $m$ disks, we have
$$|\partial \D| \ge (1-2D-\epsilon)\cdot nm$$
with probability $1-O(e^{-n^c})$
\end{theorem}

Here the hardest part is to show that the same $\epsilon$ works for van Kampen diagrams of
{\em arbitrary} size.

\begin{theorem}[Ollivier, \cite{Ollivier_small}, Cor.~3]\label{theorem:Ollivier_delta}
Let $G$ be a random group at density $D$. Then the hyperbolicity constant $\delta$
of the presentation satisfies $\delta \le 4n/(1-2D)$ with probability $1-O(e^{-n^c})$.
\end{theorem}

\begin{theorem}[Ollivier, \cite{Ollivier_small}, Thm.~6]\label{theorem:Ollivier_Greedlinger}
Let $G$ be a random group at density $D$. Then for any positive $\epsilon$, and
for any reduced van Kampen diagram $\D$ with at least two faces, there are at least
two faces which have a (connected) piece on $\partial \D$ of length at least
$n(1-5D/2 -\epsilon)$, with probability $1-O(e^{-n^c})$.
\end{theorem}

The statements of theorems in Ollivier's paper do not make the estimate of probability
(as a function of $n$) explicit; however these estimates are straightforward to
derive from his methods (and in any case, we do not use them in the sequel).

\subsection{Convexity}

We now indicate how to use small cancellation arguments to find a surface subgroup at any
$D<2/7$. This is proved by a counting argument (Lemma~\ref{lemma:density_convexity}), 
which is a model for the general case $D<1/2$ (proved in Theorem~\ref{theorem:surface_random_group}).

Pick one relator $r$, and build a beaded surface $S$ as in \S~\ref{section:one_relator_group}
whose spine is trivalent and with every edge of length $\ge L$ for some large (fixed) $L$.

\begin{lemma}\label{lemma:density_convexity}
Fix $D$. Then for any $\alpha > D$, a beaded surface $S$ constructed by the method of 
\S~\ref{section:one_relator_group} is $\alpha$-convex, with probability
$1-O(e^{-n^c})$.
\end{lemma}
\begin{proof}
By Lemma~\ref{lemma:no_long_piece} for any positive $\beta$, the spine $Y$ of $S$
has no subword of length $\beta n$ in common with the relator $r$ except for
subwords occurring in $\partial S(Y)$. For any $\alpha>0$ there are $O(2^{\alpha n/L})$
paths in $Y$ of length $\alpha n$. Define $\beta'=\log(2)\alpha/L$ so that
$e^{\beta' n}=2^{\alpha n/L}$, and note that by taking $L$ sufficiently large, we can make
$\beta'$ as small as we want.

There are $(2k-1)^{\alpha n}$ reduced words of length $\alpha n$,
and a random relator $r'$ contains $2n$ subwords of this length counting inverses
(which is polynomial in $n$, and therefore
is absorbed into the exponential terms in our estimates), so a random relator $r'$ 
has probability $O(e^{\beta' n -\log(2k-1)\alpha n})$ of having a subword of length $\alpha n$ 
in common with $Y$. If $\alpha > D$ and $\beta' < \log(2k-1)(\alpha -D)$ then no relator
$r'\ne r$ has a subword of length $\alpha n$ in common with $Y$, with probability $O(e^{-n^c})$.
\end{proof}

We deduce by Lemma~\ref{lemma:injective_surface} that a random group contains a surface
subgroup at any $D<1/12$. However, Theorem~\ref{theorem:Ollivier_Greedlinger} already 
improves this to $D<2/7$.

\subsection{van Kampen disks}

Our strategy will be to show that the existence of a certain kind of van Kampen
disk $\D$ with boundary a cyclically reduced word in $Y$ essential in $S$ gives
rise to a contradiction. Suppose that $\gamma\subset Y$ is an essential loop in $S$
whose image is trivial in $G$, so that there is some van Kampen diagram $\D$
with boundary $\gamma$. If some face in $\D$ has boundary $r$ or $r^{-1}$, and
if this face has a segment of length more than $\beta n$ in common with $\gamma$, then this face
agrees with a disk of $S$, and we can find a smaller van Kampen diagram $\D'$ by
pushing across this disk. So in the sequel we will only consider 
loops $\gamma \subset Y$ essential in $S$ 
bounding van Kampen disks $\D$ which cannot be simplified by such a move.
We call such a van Kampen disk {\em efficient}.

The following Lemma is standard.

\begin{lemma}[Short shortcut]\label{lemma:short_shortcut}
Let $G$ be a hyperbolic group with a presentation with respect to which it is
$\delta$-hyperbolic. Let $\Gamma$ be a cyclically reduced word in the generators
which is trivial in $G$. Then there is a van Kampen disk $\D$ with 
$|\partial \D| \le 18\delta$ 
and a connected subpath $\gamma \subset \partial \D$ with $\gamma \subset \Gamma$
and $|\gamma|> |\partial \D|/2$.
\end{lemma}
Note that if $\gamma' = \partial \D - \gamma$ then $|\gamma'|<|\gamma|$. In other words,
$\gamma'$ is a {\em shortcut}; hence the terminology.
\begin{proof}
In any $\delta$-hyperbolic path metric space, for any $k>8\delta$, 
a $k$-local geodesic (i.e.\/ a $1$-manifold for which every subpath of length
at most $k$ is a geodesic) 
is a (global) $(\frac{k+4\delta}{k-4\delta},2\delta)$-quasigeodesic; see \cite{Bridson_Haefliger},
Ch.~III. H, 1.13 p.~405. The loop $\Gamma$ starts and ends at the same point, and
is therefore not a $k$-local geodesic for $k\ge 9\delta$. Therefore some segment
of length at most $9\delta$ is not geodesic, and it cobounds a van Kampen disk $\D$
with an honest geodesic. 
\end{proof}

We deduce the following corollary:

\begin{lemma}\label{lemma:small_diagram_shortens}
Suppose that $S$ is a beaded surface which is not $\pi_1$-injective.
Then there are constants $C$ and $C'$ depending only on $D<1/2$,
a geodesic path $\gamma$ in the spine $Y$ of length
at most $Cn$, and a van Kampen diagram $\D$ containing at most $C'$ faces
so that $\gamma \subset \partial \D$ and $|\gamma|>|\partial \D|/2$.
\end{lemma}
\begin{proof}
Theorem~\ref{theorem:Ollivier_delta} says that $\delta \le 4n/(1-2D)$,
so by Lemma~\ref{lemma:short_shortcut} it follows that there is such a disk $\D$
with boundary of length at most $72n/(1-2D)$.
On the other hand, by Theorem~\ref{theorem:Ollivier_isoperimetric_theorem}
we know
$72n/(1-2D) \ge |\partial \D| \ge (1-2D-\epsilon)\cdot nC'$
where $C'$ is the number of faces; in particular, $C'$ is {\em bounded} in
terms of $D$ (and {\em independent} of $n$).
\end{proof}

The fact that $C$ and $C'$ can be chosen {\em independent} of
$n$ (but depending on $D<1/2$ of course) is crucial for our purposes.

\subsection{Surfaces in random groups}

We can now prove the main theorem of the paper.

\begin{theorem}[Surfaces in random groups]\label{theorem:surface_random_group}
A random group of length $n$ and density $D<1/2$ contains a surface subgroup with 
probability $1-O(e^{-n^c})$. In fact, it contains $O(e^{n^c})$ surfaces of genus $O(n)$.
Moreover, these surfaces are quasiconvex.
\end{theorem}
\begin{proof}
Pick one relation $r$ and build a beaded surface $S$ by the method of \S~\ref{section:one_relator_group}.
We have already shown that the $1$-skeleton $Y$ of $S$ does not contain a path of length
$\beta n$ for any fixed positive $\beta$ 
in common with $r$ or $r^{-1}$ except for paths in the boundary
of a disk, with the desired probability. 

Suppose $S$ is not $\pi_1$-injective. Then by Lemma~\ref{lemma:small_diagram_shortens}
there is an efficient van Kampen disk with boundary an essential loop in $Y$,
containing a subdisk $\D$ with at most $C'$ faces, and at least half of
its boundary equal to some path $\gamma$ in the spine $Y$. 
We want to show that the existence of such a van
Kampen disk is very unlikely, for fixed $D<1/2$, and for $n$ sufficiently big.

Fix a combinatorial type for the diagram. Then there are at most polynomial in $n$
choices of edge lengths for the edges in the diagram. Choose a collection of edge
lengths. Let $m\le C'$ be the number of faces.

We estimate the probability that there is a way to label each face with a relator
or its inverse compatible with some $\gamma$. We express the count in terms of
{\em degrees of freedom}, measured multiplicatively, as powers of $(2k-1)$
(Gromov uses the terms {\em density} and {\em codensity}; 
see the discussion in \cite{Gromov_asymptotic} pp. 269--272 expanded 
at length in \cite{Ollivier}).

First, the choice of $\gamma$ itself gives $n\beta'$ degrees of freedom, where
$\beta'=\log(2)\alpha/L$, and $|\gamma|=\alpha n$, as in the proof of
Lemma~\ref{lemma:density_convexity}. Since $\alpha\le C'/2$ is an absolute
constant depending only on $D$, by choosing $L$ big enough we can make $\beta'$ as small as
desired (we only really need $\beta' < 1/2-D$), 
and therefore we may effectively neglect it in what follows.

Next we consider the disks with boundary label $r$ or $r^{-1}$. Since the original disk
was efficient, no face of $\D$ with boundary label $r$ or $r^{-1}$ has a segment of
more than $\beta n$ in common with $\gamma$. Furthermore, a fixed random relation $r$ of 
length $n$ will have no piece in common with itself or its inverse
of length $\epsilon'' n$ for any positive $\epsilon''$, with probability $1-O(e^{-Cn})$; if we
take $\epsilon'' = \beta$ for simplicity, we deduce that in a reduced diagram,
no two disks with boundary label $r$ or $r^{-1}$ share a segment of their boundary in common
of length more than $\beta n$.

In a van Kampen disk with at most $m$ faces, the boundary of each face 
is decomposed into at most $m$ segments, each of which is shared with another
face or with the boundary. Thus each face with boundary
label $r$ or $r^{-1}$ has at most $m\beta n$ of its boundary in common with $\gamma$ or with
other faces with label $r$ or $r^{-1}$. Let $\D'$ be the subdiagram obtained by cutting
out the faces with boundary label $r$ or $r'$, and let $\gamma'\subset \partial\D'$ be the union
of $\gamma \cap \partial\D'$ with $\partial \D' - \partial \D$. Finally, let $m'$ be the
number of faces in $\D'$. Taking $\beta$ sufficiently small, we can assume that
$m\beta < 1/2$, and therefore $|\gamma'|\ge|\gamma|$ so that $|\gamma'|\ge |\partial\D'|/2$ 
and $m'\le m$, with equality if and only if $\D'=\D$.

Each remaining choice of face gives $nD$ degrees of freedom, and each segment in the interior
of length $\ell$ imposes $\ell$ degrees of constraint. Similarly, 
$\gamma'$ imposes $|\gamma'|$ degrees of constraint. Let $\I$ denote the union of interior
edges. Then $|\partial \D'|+2|\I|=nm'$ so $|\gamma'|+|\I|\ge nm'/2$ because $|\gamma'|\ge |\partial \D'|/2$.
On the other hand, the total degrees of freedom is $nm'D+n\beta'<nm'/2$, so no assignment
is possible, with probability $1-O(e^{-n^c})$. Summing the exceptional cases over
the polynomial in $n$ assignments of lengths, and the {\em finite} number of
combinatorial diagrams, we see that $S$ is injective, with probability 
$1-O(e^{-n^c})$ for some $c$ depending on $D$ (and going to $0$ as $D \to 1/2$).

In fact, the same argument implies that every geodesic path in the 1-skeleton of
$Y$ is actually quasigeodesic in $K$, by the proof of
Lemma~\ref{lemma:short_shortcut}. Explicitly, for any fixed $k>9\delta$ we can 
repeat the argument above with $C'$ replaced by $C'k/9\delta$, and deduce
that a geodesic in the 1-skeleton of (the universal cover of) 
$S$ is mapped to a $k$-local geodesic in (the universal cover of) $K$. 
The theorem follows.
\end{proof}

\begin{remark}
Since we may take $k>9\delta$ arbitrarily large, the estimate in the proof of
Lemma~\ref{lemma:short_shortcut} actually shows that beaded surfaces can be taken to be 
$(1+\epsilon)$-quasiconvex for any fixed $\epsilon>0$.
\end{remark}

\begin{remark}\label{remark:homologically_essential}
The surfaces we build are homologically trivial, since they map nontrivially over only one
disk bounded by a relator $r$, and with total degree $0$. As remarked in the introduction,
because the surfaces produced by the Thin Fatgraph Theorem are {\em orientable}, a modification
of our construction produces homologically essential surfaces.

If $n$ is even, a random reduced word of length $n$
in a free group of rank $k$ is homologically trivial with probability $O(n^{-k/2})$. Since there
are $(2k-1)^{nD}$ relators, there are an enormous number of such homologically trivial relators, and
we can try to build a surface mapping over the associated disk with degree $1$ (and therefore being
homologically essential in $G$). Evidently, the only obstruction to finding such surfaces is to build
a bead decomposition as in Lemma~\ref{lemma:bead_decomposition} where all the $B_i$ are
homologically trivial, while still preserving the property that correlations between
distinct $B_i$ decay exponentially fast. The probability that the naive construction of a bead
decomposition (as in the lemma) applied to a random word will have this property is 
$(n^{-k/2})^{n^\delta}$, which is subexponential in $n$, so many of the relators will have this
property, and we can build many homologically essential surfaces of genus $O(n)$. If $n$ is odd
we can build a similar (homologically essential) beaded surface from two (judiciously chosen) relators.
\end{remark}

\begin{remark}\label{remark:Gromov_norm}
The surfaces we build have genus $O(n)$ (or $o(n)$ for rank 2), 
and it is natural to wonder if this is the best possible. 
We conjecture not; in fact we conjecture that
the smallest genus injective surfaces in random groups are of genus
$O(n/\log{n})$ (at any density $D<1/2$).

In fact, Thm.~4.16 of \cite{Calegari_Walker_RR} gives a
precise estimate of the geometry of the {\em Gromov norm} on $H_2(G;\R)$.
Let $V$ be the vector space with the relators $r_i$ as basis, and let
$W$ be the kernel of the natural map $V \to H_1(F_k;\R)$. Then we can identify
$W$ with $H_2(G;\R)$, by Mayer--Vietoris. 
The vector space $W$ inherits an $L^1$ norm from $V$
with respect to its given basis. The Gromov norm on $W$ (on random subspaces
of fixed dimension) is (with overwhelming probability) proportional
to this $L^1$ norm, with constant of proportionality $2\log(2k-1)n/3\log(n)$,
up to a multiplicative error of size $1+o(1)$ (there is a factor of 4 relative
to the statement of Thm.~4.16 of \cite{Calegari_Walker_RR}; this factor of 4
reflects the difference between the Gromov norm and the so-called {\em scl norm}).
It seems plausible that classes in $H_2(G;\Q)$ should be projectively 
represented by norm-minimizing surfaces; such surfaces will necessarily be 
injective. Again, it seems likely that one should not need to pass to a very
big multiple of a class to find an extremal surface (at least for some classes);
so there should be injective surfaces of genus $O(n/\log(n))$. We strongly
suspect this order of magnitude is sharp.
\end{remark}

\end{document}